\documentclass[a4paper]{article}

\usepackage[utf8]{inputenc}
\usepackage[T1]{fontenc}
\usepackage{lmodern}

\usepackage{cite}

\usepackage{hyperref} 
\usepackage{amssymb,amsmath,amsthm}
\usepackage[english]{babel}
\usepackage{todonotes}
\usepackage{verbatim} 
\usepackage{enumerate}
\usepackage{mathtools}
\usepackage{bbm}

\newcommand{\cA}{\mathcal{A}}
\newcommand{\cB}{\mathcal{B}}
\newcommand{\cC}{\mathcal{C}}
\newcommand{\cD}{\mathcal{D}}
\newcommand{\cE}{\mathcal{E}}
\newcommand{\cF}{\mathcal{F}}

\newcommand{\cK}{\mathcal{K}}
\newcommand{\cL}{\mathcal{L}}
\newcommand{\cM}{\mathcal{M}}
\newcommand{\cN}{\mathcal{N}}

\newcommand{\cP}{\mathcal{P}}

\newcommand{\cS}{\mathcal{S}}

\newcommand{\cX}{\mathcal{X}}
\newcommand{\cY}{\mathcal{Y}}

\newcommand{\fS}{\mathfrak{S}}

\newcommand{\bE}{\mathbb{E}}

\newcommand{\bL}{\mathbb{L}}

\newcommand{\bN}{\mathbb{N}}

\newcommand{\bQ}{\mathbb{Q}}
\newcommand{\bR}{\mathbb{R}}

\newcommand{\PR}{\mathbb{P}}

\newcommand{\bONE}{\mathbbm{1}}

\newcommand{\dd}{ \mathrm{d}}

\renewcommand{\epsilon}{\varepsilon}

\newcommand{\vn}[1]{\left| \! \left| #1\right| \! \right|} 
\newcommand{\tn}[1]{\left| \! \left| \! \left|#1\right| \! \right| \! \right|}

\newcommand{\ip}[2]{\langle #1,#2\rangle}

\numberwithin{equation}{section}

\newtheorem{theorem}{Theorem}[section]
\newtheorem{lemma}[theorem]{Lemma}
\newtheorem{proposition}[theorem]{Proposition}
\newtheorem{corollary}[theorem]{Corollary}

\theoremstyle{definition}
\newtheorem{definition}[theorem]{Definition}

\newtheorem{remark}[theorem]{Remark}

\newtheorem{example}[theorem]{Example}
\newtheorem{condition}[theorem]{Condition}

\begin{document}

\title{Large deviations of the trajectory of empirical distributions of Feller processes on locally compact spaces}

\author{
\renewcommand{\thefootnote}{\arabic{footnote}}
Richard C. Kraaij
\footnotemark[1]
}

\footnotetext[1]{
Fakultät für Mathematik, Ruhr-University of Bochum, Postfach 102148, 
44721 Bochum, Germany, E-mail: \texttt{richard.kraaij@rub.de}.
}

\maketitle

\begin{abstract} 
We study the large deviation behaviour of the trajectories of empirical distributions of independent copies of time-homogeneous Feller processes on locally compact metric spaces. Under the condition that we can find a suitable core for the generator of the Feller process, we are able to define a notion of absolutely continuous trajectories of measures in terms of some topology on this core. Also, we define a Hamiltonian in terms of the linear generator and a Lagrangian as its Legendre transform.

We prove the large deviation principle and show that the rate function can be decomposed as a rate function for the initial time and an integral over the Lagrangian, finite only for absolutely continuous trajectories of measures.

We apply this result for diffusion and Lévy processes on $\bR^d$, for pure jump processes with bounded jump kernel on arbitrary locally compact spaces and for discrete interacting particle systems.
For diffusion processes, the theorem partly extends the Dawson and Gärtner theorem for non-interacting copies in the sense that it only holds for time-homogeneous processes, but on the other hand it holds for processes with degenerate diffusion matrix.
\end{abstract}

 



\section{Introduction}

Dawson and Gärtner \cite{DG87} proved the large deviation principle for the trajectory of empirical distributions of weakly interacting copies of diffusion processes. Additionally, they proved that the rate function can be decomposed as an entropy term for the large deviations at time zero and an integral over a quadratic Lagrangian, depending on position and speed. Recently, new proofs have been given using various methods and based on varying assumptions in \cite{BDF12,FK06,FaMa14}.

Similar results for Markov jump-processes has been given in \cite{DjKa95,DuRaWu16,Fe94a,Fe94b,Kr16b,Le95}. Additionally, \cite{MNW08} study the large deviations of trajectories of the empirical distributions together with the empirical flow of a finite state space Markov jump process and give a Lagrangian form of the rate function.

These two sets of results raise the question whether a context independent approach is possible to prove large deviations, in the space $D_{\cP(E)}(\bR^+)$, the Skorokhod space of $\cP(E)$ valued trajectories, for trajectories of weakly interacting, or even independent copies of processes on some space $E$ with a rate function of `Lagrangian' form:
\begin{equation} \label{eqn:intro_rate_function}
I(\nu) := \begin{cases}
I_0(\nu(0)) + \int_0^\infty \cL(\nu(s),\dot{\nu}(s))\dd s & \text{if } \nu \text{ is absolutely continuous}, \\
\infty & \text{otherwise}.
\end{cases} 
\end{equation}
For independent copies of processes, a large deviation principle(LDP) for the empirical averages can be obtained via Sanov's theorem and the contraction principle. Thus, the main goal of this paper is to rewrite this contracted rate-function in the form \eqref{eqn:intro_rate_function} in a unified way that allows for a large class of state spaces and processes including e.g. Lévy processes or independent copies of whole interacting particle systems\cite{Li85}. 

\smallskip

The results in this paper should be compared to large deviation principle, and the representation of the rate function, for the empirical process $n^{-1} \sum_{i \leq n} \delta_{X^i}$ on the space $\cP(D_E(\bR^+))$ in \cite{QuReVa99,Se16}. In the first paper, the $X^i$ are particles on a discrete lattice interacting via an exclusion rule and in the second paper the $X^i$ are Brownian particles with local interaction. In both cases, it is shown that the rate function has a conditional structure, composed of two parts. The first part is the rate function for the trajectory of empirical measures, in Lagrangian form, as described above. The second part is the path-space relative entropy of the measure with respect to a specifically tilted Markov process that has the correct marginals. In \cite{Se16}, the question is raised whether such results are equally robust, but this question goes beyond the results in this paper, both in terms of the interaction as in the space for which the LDP is stated.

\smallskip

To give a uniform proof of \eqref{eqn:intro_rate_function}, one cannot use any explicit structure of the underlying process, so we use the functional analytic structure underlying the Girsanov transformation that has been extensively studied in \cite{FK06}. Compared to \cite{FK06}, the focus of this paper is different. The independence assumption implies that the large deviation principle can be proven via Sanov's theorem and the contraction principle. Therefore, the main problem that is being addressed in this paper is the expression of the rate function in a Lagrangian form.

To obtain this Lagrangian form, we study the non-linear semigroup $\{V(t)\}_{t \geq 0}$ on $C_0(E)$ of conditional log Laplace transforms, defined by $V(t)f = \log S(t)e^f$, where $S(t)$ is the linear transition semigroup of the Feller process on a locally compact Polish space $E$. The main technical step in this paper is to show that the lift of the semigroup $V(t)$  to $C(\cP(E))$ equals a variational Nisio semigroup $\mathbf{V}(t)$:
\begin{equation*}
\mathbf{V}(t) G(\mu) = \sup_{\nu \in \cA \cC_\mu} \left\{G(\nu(t)) - \int_0^t \cL(\nu(s),\dot{\nu}(s)) \dd s \right\},
\end{equation*}
where $G \in C(\cP(E))$ and $\cA\cC_\mu$ is the space of `absolutely continuous' $\cP(E)$-valued trajectories that start in $\mu$. The definition of the Nisio-semigroup poses us with two problems. First, we need a way to define absolutely continuous trajectories of measures, and secondly, we need a way to define a Lagrangian. To this end, we assume the existence of a suitable topology on a core of the generator $(A,\cD(A))$ of the Feller process. This topology can then be used to define absolute continuity and the Lagrangian can be defined as the Legendre transform of $H$, $Hf = e^{-f}Ae^f$, with respect to the duality of $D$ with $D'$. The equality of $V(t)$ and $\mathbf{V}(t)$ is proven using resolvent approximation arguments and Doob-transform techniques.

\smallskip

The paper is organised as follows. We start out in Section \ref{sect:results} with the preliminaries and state the two main results. Theorem \ref{The:LDP1} gives, under the condition that the processes solves the martingale problem, the large deviation principle. Under the condition that there exists a suitable core for the generator of the process, Theorem \ref{The:LDP2} gives the decomposition of the rate function.

In Section \ref{section:control_theory_approach}, we study functional analytic properties of the generator, its non-linear counterpart $H$ and the Lagrangian $\cL$. Additionally, we show that $H$ is intimately related to the Girsanov transforms of the Markov process with generator $A$. In Section \ref{section:proof_of_ldp2}, we prove Theorem \ref{The:LDP2}. In particular, we introduce the Nisio semigroup $\mathbf{V}(t)$ in terms of absolutely continuous trajectories and the Lagrangian, and show that it equals the non-linear semigroup $V(t)$.

In Section \ref{section:examples}, we give four examples where Theorem \ref{The:LDP2} applies. We start with diffusion processes. After that, we consider Lévy processes and Markov jump processes. Finally, we check the conditions for spatially extended interacting particle systems of the type that are found in Liggett \cite{Li85}.

\section{Preliminaries and main results} \label{sect:results}

We start with some notation. Let $(E,d)$ be a complete separable metric space with Borel $\sigma$-algebra $\cE$. $\cM(E)$ is the set of Borel measures of bounded total variation on $E$ be equipped with the weak topology and $\cP(E)$ is the subset of probability measures. We denote with $D_E(\bR^+)$ the Skorokhod space of $E$ valued c\`{a}dl\`{a}g paths\cite[Section 3.5]{EK86}, $\bR^+ = \left[\left.0,\infty\right)\right.$. We write $\ip{f}{\mu}$ for the integral of $f \in C_b(E)$ with respect to $\mu \in \cM(E)$.

We define the relative entropy $H(\mu \, | \, \nu)$ of $\mu$ with respect to $\nu$ by
\begin{equation} \label{definition:relative_entropy}
H(\mu \, | \, \nu) = 
\begin{cases}
\int \log \frac{\dd \mu}{\dd \nu} \dd \mu & \text{if } \mu << \nu \\
\infty & \text{otherwise}.
\end{cases}
\end{equation}
On $E$, we have a time-homogeneous Markov process $\left\{X(t)\right\}_{t \geq 0}$ given by a path space measure $\PR$ on $D_E(\bR^+)$. Let $X^1, X^2, \dots$ be independent copies of $X$ and let $P$ the measure that governs these processes. We look at behaviour of the sequence $L_n := \left\{L_n^{X(t)}\right\}_{t \geq 0}$, 
\begin{equation*}
L_n^{X(t)} := \frac{1}{n} \sum_{i=1}^n \delta_{\{X^i(t)\}},
\end{equation*} 
under the law $P$. $L_n$ takes values in $D_{\cP(E)}(\bR^+)$, the Skorokhod space of paths taking values in $\cP(E)$. We also consider $C_{\cP(E)}(\bR^+)$ the space of continuous paths on $\cP(E)$ with the topology inherited from $D_{\cP(E)}(\bR^+)$.

\smallskip

We say that $L_n$ satisfies the large deviation principle(LDP) on $D_{\cP(E)}(\bR^+)$ with lower semi-continuous rate function $I : D_{\cP(E)}(\bR^+) \rightarrow [0,\infty]$ if for every open set $A$
\begin{equation*}
\liminf_{n \rightarrow \infty} \frac{1}{n} \log P[L_n \in A] \geq - \inf_{\mu \in A} I(\mu)
\end{equation*}
and for every closed set $B$
\begin{equation*}
\limsup_{n \rightarrow \infty} \frac{1}{n} \log P[L_n \in B] \leq - \inf_{\mu \in B} I(\mu).
\end{equation*}
$I$ is called good if its level sets $\{I\leq c\}$ are compact.

\smallskip

Suppose that $A : \cD(A) \subseteq C_b(E) \rightarrow C_b(E)$, is a linear operator with a domain that separates points: for every $x,y \in E$, there exists a $f$ in this set such that $f(x) \neq f(y)$. We say that $X$ solves the martingale problem for $(A,\cD(A))$ with starting measure $\PR_0$, if $\PR_0$ is the law of $X(0)$ and if for every $f \in \mathcal{D}(A)$, 
\begin{equation*}
f(X(t)) - f(X(0)) - \int_0^t Af(X(s)) \dd s
\end{equation*}
is a martingale for the natural filtration $\{\cF_t\}_{t \geq 0}$ given by $\cF_t = \sigma(X(s) \, | \, s \leq t)$.

In Appendix \ref{section:LDP_on_path_space}, we obtain the following preliminary result.

\begin{theorem} \label{The:LDP1}
Let $X$, represented by the measure $\PR$ on $D_E(\bR^+)$ solve the martingale problem for $(A,\cD(A))$ with starting measure $\PR_0$. Then, the sequence $L_n$ satisfies the large deviation principle with good rate function $I$, which is given for $\nu = \left\{\nu(t)\right\}_{t \geq 0} \in D_{\cP(E)}(\bR^+)$ by
\begin{equation*}
I(\nu) = \begin{dcases}
H(\nu(0) \, | \, \PR_0) + \sup_{\{t_i\}} \sum_{i=1}^k I_{t_i- t_{i-1}}(\nu(t_i) \, | \, \nu(t_{i-1})) & \text{if }\nu \in C_{\mathcal{P}(E)}(\bR^+) \\
\infty & \text{otherwise},
\end{dcases}
\end{equation*} 
where $\{t_i\}$ is a finite sequence of times: $0 = t_0 < t_1 < \dots < t_k$. For $s \leq t$, we have $I_{t}(\nu_2 \, | \, \nu_1) := \sup_{f \in C_b(E)} \left\{ \ip{f}{\nu_2} - \ip{V(t) f}{\nu_1} \right\}$, where $V(t)f(x) := \log \bE\left[e^{f(X(t))} \, \middle| \, X(0) = x \right]$. 
\end{theorem} 

For further results, we introduce some additional notation. For a locally convex space $(\cX,\tau)$, we write $\cX'$ for its continuous dual space. For $x \in \cX$ and $x' \in \cX'$, we write $\ip{x}{x'} := x'(x) \in \bR$ for the natural pairing between $x$ and $x'$. For two locally convex spaces $\cX, \cY$ and a continuous linear operator $T : \cX \rightarrow \cY$, we write $T' : \cY' \rightarrow \cX'$ for the adjoint of $T$, which is uniquely defined by $\ip{x}{T'(y')} = \ip{Tx}{y'}$, see for example Treves \cite[Chapter 19]{Tr67}. For a neighbourhood $\cN$ of $0$ in $\cX$, we define the polar of $\cN^\circ \subset \cX'$ by
\begin{equation} \label{eqn:polar_def}
\cN^\circ := \left\{u \in \cX' \, \middle| \,|\ip{x}{u}| \leq 1 \text{ for every } x \in \cN \right\}.
\end{equation}
We say that a locally convex space $\cX$ is barrelled if every barrel is a neighbourhood of $0$. A set $S$ is a barrel if it is convex, balanced, absorbing and closed. $S$ is balanced if we have the following: if $x \in S$ and $\alpha \in \bR$, $|\alpha| \leq 1$ then $\alpha x \in S$. $S$ is absorbing if for every $x \in \cX$ there exists a $r \geq 0$ such that if $|\alpha| \geq r$ then $x \in \alpha S$. Barrelled spaces are of importance in view of this paper, because they allow for a well-defined integration theory on the dual space. We state the main result in this direction in Appendix \ref{section:FAappendix}. For example, Banach, Fr\'{e}chet and LF(limit Fr\'{e}chet) spaces are barrelled \cite[Chapter 33]{Tr67}. The space of Schwartz functions is Fr\'{e}chet and the space $C_c^\infty(\bR^d)$ with its usual topology is LF.

\smallskip

To rewrite the rate function obtained in Theorem \ref{The:LDP1}, we restrict to locally compact metric spaces $(E,d)$ and we consider the situation where $S(t)f(x) = E[f(X(t)) \, | \, X(0)=x]$ is a strongly continuous semigroup on the space $(C_0(E),\vn{\cdot})$: for every $t \geq 0$, the map $S(t) : (C_0(E),\vn{\cdot}) \rightarrow  (C_0(E),\vn{\cdot})$ is continuous, and for every $f \in C_0(E)$, the trajectory $t \mapsto S(t)f$ is continuous in $(C_0(E),\vn{\cdot})$.

Let $(A,\cD(A))$ be the generator of the semigroup $S(t)$. It is a well known result that $X$ solves the martingale problem for $(A,\cD(A))$ \cite[Proposition 4.1.7]{EK86}, so the above result holds for the process $\{X(t)\}_{t \geq 0}$.

Our goal is to rewrite $I$ as
\begin{equation*}
I(\nu) = H(\nu(0) \, | \, \PR_0) + \int_0^t \cL(\nu(s),\dot{\nu}(s))\dd s
\end{equation*}
for a trajectory $\nu$ of probability measures that is absolutely continuous in some sense. Thus our first problem is to define differentiation in a context for which no suitable structure on $E$ or $\cP(E)$ is known. Therefore, we will have to tailor the definition of differentiation to the process itself. 
Suppose that $\mu(t)$ is the law of $X(t)$ under $\PR$. Then we know that $t \mapsto \mu(t) = S(t)' \mu(0)$ is a weakly continuous trajectory in $\cP(E)$, so can ask whether for $f \in \cD(A)$ the trajectory $t \mapsto \ip{f}{\mu(t)}$ is differentiable as a function from $\bR^+ \rightarrow \bR$:
\begin{equation} \label{eqn:weak_differentiation_intro}
\frac{\partial}{\partial t} \ip{f}{\mu(t)} = \frac{\partial}{\partial t} \ip{S(t)f}{\mu(0)} = \ip{S(t)A f}{\mu(0)} = \ip{A f}{\mu(t)}. 
\end{equation}
Thus our candidate for $\dot{\mu}(t)$ is $A' \mu(t)$, which is problematic because $(A,\cD(A))$ could be unbounded. To overcome this, and other problems, we introduce two sets of conditions on $(A,\cD(A))$.

\smallskip

Recall that $D$ is a core for $(A,\cD(A))$ if $D$ is dense in $(C_0(E),\vn{\cdot})$ and if for every $f \in \cD(A)$, we can find a sequence $f_n \in D$ such that $f_n \rightarrow f$ and $Af_n \rightarrow Af$. For general properties of cores see \cite[Chapter 1]{EK86} or \cite[Chapter 2]{EN00}.

\begin{condition} \label{condition:D_algebra_closed_smooth}
There exists a core $D \subseteq \cD(A)$ that satisfies
\begin{enumerate}[(a)]
\item $D$ is an algebra, i.e. if $f,g \in D$ then $fg \in D$,
\item if $f \in D$ and $\phi : \bR \rightarrow \bR$ a smooth function on the closure of range of $f$, then $\phi \circ f - \phi(0) \in D$, 
\end{enumerate}
In the case that $E$ is compact, $C_0(E) = C(E)$, then (b) can be replaced by
\renewcommand{\descriptionlabel}[1]{\hspace{\labelsep}#1}
\begin{description}
\item[(b')] if $f \in D$ and $\phi : \bR \rightarrow \bR$ a smooth function on the range of $f$, then $\phi \circ f \in D$.
\end{description}
\end{condition}

Under Condition \ref{condition:D_algebra_closed_smooth}, we define the operator $H : D \rightarrow C_0(E)$ and for every $g \in D$ the operator $A^g : D \rightarrow C_0(E)$ by
\begin{equation*}
Hf = e^{-f} A e^f, \qquad A^g f = e^{-g} A(fe^g) - (e^{-g}Ae^g)f.
\end{equation*}
These definitions follow \cite{FK06} and are at the basis of a functional analytic approach for studying the Girsanov transform. If $E$ is non-compact, these definitions needs some care as $e^f \notin C_0(E)$. This can be solved by looking at the one-point compactification of $E$, see Section \ref{section:semigroupV}. In Section \ref{section:control_theory_approach}, we will show that $\{V(t)\}_{t \geq 0}$ is a non-linear semigroup on $C_0(E)$ which has a generator that extends $H$. The operators $A^g$ are generators of Markov processes with law $\bQ^g$ on $D_E([0,t])$ that are obtained from $\PR$ by 
\begin{equation} \label{eqn:tilting_procedure_Ag}
\frac{\dd \bQ^g_t}{\dd \PR_t}(X) = \exp\left\{g(X(t)) - g(X(0)) - \int_0^t Hg(X(s)) \dd s \right\},
\end{equation}
where $\PR_t$ and $\bQ^g_t$ are the measures $\PR$ and $\bQ^g$ restricted to times up to $t$, see Proposition \ref{proposition:Girsanov_transform} below.

\begin{condition}[Conditions on the core] \label{condition:topology_on_D}
$D$ satisfies Condition \ref{condition:D_algebra_closed_smooth} and there exists a topology $\tau_D$ on $D$ such that
\begin{enumerate}[(a)]
\item $(D,\tau_D)$ is a separable barrelled locally convex Hausdorff space.
\item The topology $\tau_D$ is finer than the sup norm topology restricted to $D$.
\item If $\phi : [a,b] \rightarrow \bR$ is smooth and such that $\phi(0) = 0$, then the map $T_\phi : D \cap \{f \in D \, | \, f(E) \subseteq [a,b]\} \rightarrow D$, defined by $T_\phi f = \phi \circ f$ is $\tau_D$ to $\tau_D$ continuous.
\item The map $A : (D,\tau_D) \rightarrow (C_0(E),\vn{\cdot})$ is continuous.
\item There exists a barrel $\cN \subseteq (D,\tau_D)$ such that for every $c > 0$, we have $\sup_{f \in c \cN} \vn{Hf} < \infty$.
\end{enumerate}
\end{condition}

Conditions (a) and (b) make sure that $(D,\tau_D)$ is well behaved as a locally convex space in relation to $C_0(E)$. Among other things, we are able to define the Gelfand integral, see Appendix \ref{section:FAappendix}.

Condition (c) will imply that not only (d) holds, but also that (d) holds for all operators $A^g$. This makes sure that we can define $A'$ and $(A^g)'$ to define the weak derivative of suitable trajectories of measures as in \eqref{eqn:weak_differentiation_intro}. In other words: if $\mu^g(t)$ is the trajectory of measures obtained by $S^g(t)' \mu(0)$, where $\{S^g(t)\}_{t \geq 0}$ is the semigroup corresponding to the change of measure in \eqref{eqn:tilting_procedure_Ag}, then $\dot{\mu}^g(t) := (A^g)' \mu^g(t) \in D'$.

\smallskip

The existence of a barrel $\cN$ such that $\sup_{f \in \cN} \vn{Hf} < \infty$ follows from (c), (d) and Lemma \ref{lemma:continuity_of_operators} below. Thus the real assumption in (e) is that one can find a single $\cN$ that works for all $c \geq 0$. This can be interpreted as a growth bound on $H$, which can be used to obtain the compactness of the level sets of $\cL$, and to obtain bounds on linear functionals in terms of the Lagrangian in Lemma \ref{lemma:norm_bounded_by_Lagrangian}. We give an example of a Markov-jump process where this condition is not satisfied due to the global unboundedness of the jump rates in Section \ref{section:example_Markov_jump_processes}.

\smallskip

Note that the barrel $\cN$ can always be replaced by the barrel $\cN^*$ obtained by adding all the constant functions $\alpha\bONE$, $\alpha \in \bR$ and then taking the convex hull. Then $\cN^*$ also satisfies (e). If $g \in c \cN^*$ then there is a $\lambda \in [0,1]$, $f \in c \cN$ and $\alpha \in \bR$ such that $g = \lambda f + (1-\lambda) \alpha \bONE$. Because the map $h \mapsto H h$ is convex(see the proof of Lemma \ref{lemma:variationalexpressionH}), we find
\begin{equation*}
\vn{Hg} \leq \vn{H(\lambda f + (1-\lambda)\alpha \bONE )} \leq \lambda \vn{Hf} + (1-\lambda)\alpha \vn{H \bONE} \leq \vn{Hf}.
\end{equation*}

Thus, we will implicitly assume that $\cN$ includes all the constant functions.

The following lemma is a consequence of Condition \ref{condition:topology_on_D} (c) and (d) and the proof is elementary.

\begin{lemma} \label{lemma:continuity_of_operators}
Let $(D,\tau_D)$ satisfy Condition \ref{condition:topology_on_D}, then the maps $\cA :(D,\tau_D) \times (D,\tau_D) \rightarrow (C_0(E),\vn{\cdot})$ given by $\Phi(g,f) = A^g f$ and the operator $H : (D,\tau_D) \rightarrow (C_0(E),\vn{\cdot})$ are continuous.

\smallskip

Let $g \in D$. As a consequence of the second statement, the map $A^g : (D,\tau_D) \rightarrow (C_0(E),\vn{\cdot})$ is continuous.
\end{lemma}

\begin{remark}
The results of this paper also hold in the case that Condition \ref{condition:topology_on_D} (c) fails as long as the conclusions of Lemma \ref{lemma:continuity_of_operators} hold. In all examples that we consider in Section \ref{section:examples}, (c) is satisfied.
\end{remark}

For the next definition we will need the Gelfand or weak* integral, which is introduced in Appendix \ref{section:FAappendix}, but the rigorous construction of this integral can be skipped on the first reading.

\begin{definition} \label{definition:absolutely_continuous}
Define $D-\cA\cC$, or if there is no chance of confusion, $\cA\cC$, the space of (weakly) absolutely continuous paths in $C_{\cP(E)}(\bR^+)$. A path $\nu \in C_{\cP(E)}(\bR^+)$ is called absolutely continuous if there exists a $(D',wk^*)$ measurable curve $s \mapsto u(s)$ in $D'$ with the following properties:
\begin{enumerate}[(i)]
\item for every $f \in D$ and $t \geq 0$ $\int_0^t |\ip{f}{u(s)}| \dd s < \infty$,
\item for every $t \geq 0$, $\nu(t) - \nu(0) = \int_0^t u(s) \dd s$ as a $D'$ Gelfand integral, i.e.
\begin{equation*}
\ip{f}{\nu(t) - \nu(0)} = \ip{f}{\int_0^t u(s) \dd s} = \int_0^t \ip{f}{u(s)} \dd s, \qquad \forall f \in D.
\end{equation*}
\end{enumerate}
We denote $\dot{\nu}(s) := u(s)$. Furthermore, we will denote $\cA\cC_{\mu}$ for the space of absolutely continuous trajectories starting at $\mu_0$, and $\cA\cC^T$ for trajectories that are only considered up to time $T$. Finally, we define $\cA\cC_{\mu}^T = \cA\cC_\mu \cap \cA\cC^T$.
\end{definition}

A direct consequence of the definition is that if $\nu \in \cA\cC$ then for almost every time $t \geq 0$ and all $f \in D$ the limit
\begin{equation*}
\lim_{h \rightarrow 0} \frac{\ip{f}{\nu(t+h)} - \ip{f}{\nu(t)}}{h} 
\end{equation*}
exists and is equal to $\ip{f}{\dot{\nu(t)}}$. This justifies the notation $u(s) = \dot{\nu}(s)$.

\begin{remark}
When we apply this definition for $D$ equal to the space of compactly supported smooth functions on $\bR^d$ with the natural inductive limit topology, a curve is absolutely continuous in the sense of Definition 4.1 in \cite{DG87} is absolutely continuous in the sense of \ref{definition:absolutely_continuous}. For a trajectory with finite Lagrangian cost, in the sense of the next theorem, the converse holds as well. See Proposition \ref{proposition:finite_lagrangian_cost_implies_strong_abs_cont} and Lemma \ref{lemma:absolutely_continuous_DG} below. 
\end{remark}

Using these definitions, we are able to improve on Theorem \ref{The:LDP1}. 

\begin{theorem} \label{The:LDP2}
Let $(E,d)$ be locally compact. Let $(A,\cD(A))$ have a core $D$ equipped with a topology $\tau_D$ such that $(D,\tau_D)$ satisfies Condition \ref{condition:topology_on_D}. Then, the rate function in Theorem \ref{The:LDP1} can be rewritten as 
\begin{equation*}
I(\nu) = \begin{cases}
H(\nu(0)\, | \, \PR_0) + \int_0^\infty \cL(\nu(s),\dot{\nu}(s)) \dd s  & \text{if } \nu \in \cA \cC \\
\infty & \text{otherwise},
\end{cases}
\end{equation*}
where $\cL : \cP(E) \times D' \rightarrow [0,\infty]$ is given by $\cL(\mu,u) := \sup_{f \in D} \left\{ \ip{f}{u} - \ip{Hf}{\mu} \right\}$.
\end{theorem}

\begin{remark}
If we restrict ourselves to $[0,T]$ instead of $\bR^+$, then we obtain 
\begin{equation*}
I^T(\left\{\nu(s)\right\}_{0 \leq s \leq T}) = \begin{cases}
H(\nu(0)\, | \, \PR_0) +  \int_0^T \cL(\nu(s),\dot{\nu}(s)) \dd s  & \text{if } \nu \in \cA \cC^T \\
\infty & \text{otherwise},
\end{cases}
\end{equation*}
by applying the contraction principle. 
\end{remark}

\begin{remark}
Note that the rate function in terms of $\cL$ can be obtained heuristically from the form of $I_t$ in Theorem \ref{The:LDP1}. Suppose that $\nu \in \cA\cC$. Then
\begin{align*}
& \frac{1}{h} I_h(\nu(t+h) \, | \, \nu(t)) \\
& \quad = \frac{1}{h}\sup_{f \in C_b(E)} \left\{\ip{f}{\nu(t+h)} - \ip{f}{\nu(t)} - \ip{V(h)f - f}{\nu(t)} \right\} \\
& \quad = \frac{1}{h}\sup_{f \in D} \left\{\ip{f}{\nu(t+h)} - \ip{f}{\nu(t)} - \ip{V(h)f - f}{\nu(t)} \right\}. 
\end{align*}
Formally interchanging the limit as $h \downarrow 0$ and taking the supremum over $f \in D$ yields $\frac{1}{h} I_h(\nu(t+h) \, | \, \nu(t)) \approx \sup_{f \in D} \ip{f}{\dot{\nu}(t)} - \ip{Hf}{\nu(t)} = \cL(\nu(t),\dot{\nu}(t))$. This argument can be put to work for continuous $\cL$ and piece-wise `continuously differentiable' trajectories via the Riemann integral and a sequence of  careful choices of times $t_0 < t_1 < \dots < t_n$. However, for arbitrary absolutely continuous trajectories it is not clear to the author how to make such an argument rigorous.
\end{remark}

For trajectories with finite Lagrangian cost, we can strengthen the absolute continuity to strong absolute continuity, in the spirit of Definition 4.1 of \cite{DG87}. 

\begin{definition} \label{definition:strongly_absolutely_continuous}
We say that a path $\nu \in C_{\cP(E)}(\bR^+)$ is strongly absolutely continuous if there exists an absolutely continuous function $H : [0,\infty) \rightarrow \bR$ such that $\sup_{f \in \cN} |\ip{f}{\nu(t)} - \ip{f}{\nu(s)}| \leq |H(t) - H(s)|$ for all $s,t \geq 0$.
\end{definition}

Note that absolute continuity is much easier to establish than strong absolute continuity. Thus, for the main proofs we will use the weak notion. We do mention strong absolute continuity, because this notion allows one to prove integration by parts formula's as in Lemma 4.3 in \cite{DG87}. In general, we have the following result that allows us to strengthen the weak notion to the strong notion.

\begin{proposition} \label{proposition:finite_lagrangian_cost_implies_strong_abs_cont}
Let $(E,d)$ be locally compact. Let $(A,\cD(A))$ have a core $D$ equipped with a topology $\tau_D$ such that $(D,\tau_D)$ satisfies Condition \ref{condition:topology_on_D}. Then we have the following two results.
\begin{enumerate}[(a)]
\item If $\gamma \in C_{\cP(E)}(\bR^+)$ is strongly absolutely continuous, then it is absolutely continuous.
\item If $\gamma \in C_{\cP(E)}(\bR^+)$ is absolutely continuous and is such that
\begin{equation*}
\int_0^\infty \cL(\gamma(s),\dot{\gamma}(s)) \dd s < \infty,
\end{equation*}
then it is strongly absolutely continuous.
\end{enumerate}
\end{proposition}

\section{A study of the operators \texorpdfstring{$V(t)$, $H$, $L$ and $A^g$}{V(t), H, L and Ag}} \label{section:control_theory_approach}

\subsection{The semigroup \texorpdfstring{$V(t)$}{V(t)} and the generator \texorpdfstring{$H$}{H}} \label{section:semigroupV}

We return to the situation that $(E,d)$ is a locally compact metric space, so that we can use semigroup theory to rewrite the rate function. 

First suppose that $E$ is non-compact. Let $E^\Delta = E \cup \{\Delta\}$ be the one-point compactification. By Lemma 4.3.2 in \cite{EK86}, $S(t)$ extends to a strongly continuous contraction semigroup on $(C(E^\Delta), \vn{\cdot})$ by setting $S^\Delta(t)f = f(\Delta) + S(t)(f - f(\Delta))$. Therefore, we can argue using the semigroup on the compact space $E^\Delta$, and then obtain the result in Theorem \ref{The:LDP2} on $E$ by Theorem 4.11 in Feng and Kurtz \cite{FK06}.

Technically, we would also need to add the constants to the core $D$ of Conditions \ref{condition:D_algebra_closed_smooth} and \ref{condition:topology_on_D}. In other words, we should consider $D^\Delta := D \oplus \bR$ with its natural topology $\tau_D^\Delta$. However, the generator $A^\Delta$ of the extended semigroup $\{S^\Delta(t)\}_{t \geq 0}$ satisfies $A^\Delta \bONE = 0$. As a consequence it also holds that $H^\Delta \bONE = 0$, and we can include the constants in a natural way into the barrel $\cN$ of Condition \ref{condition:topology_on_D} (e). By Lemma 3.11 below, this implies that the space $U$ of speeds that have finite Lagrangian cost is a subspace of $D'$. So indeed Theorem \ref{The:LDP2} holds with respect to the core $(D,\tau_D)$ of the generator $(A,\cD(A))$ instead of the core $(D^\Delta,\tau_D^\Delta)$ of $(A^\Delta,\cD(A^\Delta))$.

\smallskip

From this point onward, we assume that $(E,d)$ is compact and that the transition semigroup $\{S(t)\}_{t \geq 0}$ is strongly continuous on $C(E)$. Let $A : \cD(A) \subseteq C(E) \rightarrow C(E)$ be the associated infinitesimal generator.

We examine $V(t)f(x) = \log S(t) e^f(x) = \log \bE\left[e^{f(X(t))} \, \middle| \, X(0) = x \right]$, $f \in C(E)$, which was defined in Theorem \ref{The:LDP1}. It is an elementary calculation to check that $\{V(t)\}_{t \geq 0}$ is a strongly continuous contraction semigroup on $C(E)$.

As in the linear case, define the generator $H$ of $\{V(t)\}_{t \geq 0}$ to be 
\begin{equation*}
Hf = \lim_{t \downarrow 0} \frac{V(t)f - f}{t}
\end{equation*}
defined for $f \in \cD(H)$, where
\begin{equation*}
\cD(H) := \left\{f \in C(E) \, \middle| \, \exists g \in C(E): \, \lim_{t \downarrow 0} \vn{\frac{V(t)f - f}{t} - g} = 0 \right\}.
\end{equation*} 

We start with an extension of the chain rule to Banach spaces. The proof is rather standard and is left to the reader.

\begin{lemma} \label{lemma:differentiation_smooth_semigroup} 
Let $\{f(s)\}_{s \in [0,\varepsilon]}$, $\varepsilon > 0$ be a collection of bounded functions $f(s) : E \rightarrow \bR$ such that $s \mapsto f(s)$ is norm continuous in $C(E)$. Additionally suppose that 
\begin{equation*}
g := \lim_{t \rightarrow 0} \frac{f(t) - f(0)}{t}
\end{equation*}
exists in norm. Denote by $S \subseteq \bR$ the union of ranges $S = \cup_s f(s)(E)$.
Let $\phi : S \rightarrow \bR$ be differentiable on $S$ and let $\phi'$ be Lipschitz continuous. Then it holds that $\frac{\dd}{\dd t} \phi(f(t))|_{t = 0} = \phi'(f(0)) g$, which should be interpreted as
\begin{equation*}
\lim_{t \rightarrow 0} \frac{\phi(f(t)) - \phi(f(0))}{t} = \phi'(f(0)) g
\end{equation*}
with respect to the sup norm.
\end{lemma}

Because for any fixed given $f \in \cD(A)$, we have that $\lim_{t \rightarrow 0} \frac{T(t)f - f}{t} = Af$, we can explicitly calculate the generator $H$ of $V(t)$ on its domain.

\begin{corollary} \label{corollary:calculationH}
For $f \in C(E)$, $e^f \in \cD(A)$ is equivalent to $f \in \cD(H)$ and if this holds, then $Hf = e^{-f}A(e^f)$.
\end{corollary}

\begin{proof}
Because $f \in C(E)$ it is bounded from below and as $S(t)$ is contractive for all $t \geq 0$, we know that $\inf_{x \in e} \inf_t S(t)e^f(x) > 0$. Thus the logarithm and its derivative are Lipschitz on the union of the ranges of $S(t)e^f$. Thus it follows that $f \in \cD(H)$ and $Hf = e^{-f}Ae^f$ by Lemma \ref{lemma:differentiation_smooth_semigroup}.

The proof in the other direction follows similarly as the exponential function and its derivative are Lipschitz on every bounded domain.
\end{proof}

We note that as a consequence of Condition \ref{condition:D_algebra_closed_smooth}, Corollary \ref{corollary:calculationH} gives us that if $f \in D$, then $f \in \cD(H)$ and $Hf = e^{-f}A e^f$. 

Because $D$ is closed under composition with smooth functions, $D$ acts as a `core' for the non-linear operator $(H,\cD(H))$.

\begin{lemma} \label{lemma:approximation_core_H}
Let $f \in \cD(H)$, then we can find a sequence of functions $g_n \in D$ such that $\vn{g_n - f} + \vn{Hg_n - Hf} \rightarrow 0$.
\end{lemma}

\begin{proof}
Because $f \in \cD(H)$, we have $e^f \in \cD(A)$ by an application of Lemma \ref{lemma:differentiation_smooth_semigroup}. $D$ is a core for $(A,\cD(A))$, so we can find $h_n \in D$ such that $\vn{e^f - h_n} + \vn{A e^f - A h_n} \rightarrow 0$. As $f$ is a bounded function, we find that $\inf_x e^f(x) > 0$. Thus, we can assume without loss of generality $\alpha := \inf_n \inf_x h_n(x) > 0$. We define $g_n := \log h_n$. $D$ is closed under composition with smooth functions, which implies that $g_n \in D$.

On $[\alpha, \infty)$ the logarithm $x \mapsto \log x$ is uniformly continuous which implies $\vn{g_n - f} \rightarrow 0$. The map $x \mapsto x^{-1}$ is also uniformly continuous on $[\alpha, \infty)$, which implies are uniformly continuous, thus we find $\vn{e^{-g_n} - e^{-f}} \rightarrow 0$. Because taking products is norm continuous, we find
\begin{equation*}
\vn{H g_n - Hf)} = \vn{e^{-g_n} A h_n - e{-f} A e^f} \rightarrow 0.
\end{equation*}
\end{proof}

We will use this operator $(H,D)$, under Condition \ref{condition:topology_on_D}, to construct a new \textit{Nisio} semigroup $\{\mathbf{V}(t)\}_{t \geq 0}$ on $C(\cP(E))$ that formally equals the semigroup $\{V(t)\}_{t \geq 0}$. This new variational semigroup will be introduced in Section \ref{section:Nisio_semigroup} below and is given by a cost optimization problem. The cost is given in terms of a Lagrangian, that we will introduce next.

\subsection{Operator duality for \texorpdfstring{$H$}{H}}

Additionally to the operator $H$, we introduce operators $A^g$ that serve as generators of tilted Markov processes obtained from $X(t)$ by the change of measure given in Equation \eqref{eqn:tilting_procedure_Ag}. We also introduce an operator $L$, that will serve as a precursor to our final Lagrangian $\cL$.

\begin{definition}
Under Condition \ref{condition:D_algebra_closed_smooth}, define the following operators for $f,g \in D$: $A^g f = e^{-g} A(fe^g) - (e^{-g}Ae^g)f$, $Lg  = A^g g - Hg$.
\end{definition}

$H$ will be called the Hamiltonian and $L$ the (pre-)Lagrangian in analogy to the Lagrangian and Hamiltonian  of classical mechanics. $A^g$ is a generator itself, as we will show below. This is also illustrated by the next two examples. We calculate $H$ and $A^g$ in the case of a Markov jump process and a standard Brownian motion.

\begin{example}
Let $E$ be a finite set and let $\left\{X(t)\right\}_{t \geq 0}$ be generated by 
\begin{equation*}
Af(x) = \sum_y r(x,y)\left[f(y)-f(x)\right],
\end{equation*}
where $r$ is some transition kernel. A calculations shows that
\begin{align*}
Hf(x) & = \sum_y r(x,y) \left[e^{f(y)-f(x)} -1 \right], \\
A^g f(x) & =  \sum_y r(x,y)e^{g(y)-g(x)} \left[f(y)-f(x) \right]. 
\end{align*}
\end{example}

\begin{example}
Let $E = \bR$, and let $\left\{X(t)\right\}_{t \geq 0}$ be a standard Brownian motion, for which the generator $A$ is given for $f \in C_c^\infty(\bR)$, i.e. smooth and compactly supported functions, by $A f(x) = \frac{1}{2} f''(x)$. $H$ and $A^g$ are given by
\begin{equation*}
Hf(x) = \frac{1}{2}f''(x) + \frac{1}{2}(f'(x))^2, \qquad A^g f(x) =  \frac{1}{2}f''(x) + f'(x)g'(x). 
\end{equation*}
\end{example}

The claim that $A^g$ is a generator is made precise by the following Girsanov transform, see Theorem 4.2 in \cite{PR02}.

\begin{proposition} \label{proposition:Girsanov_transform}
Suppose that $g \in D$ and consider the measure $\bQ^g_{[0,T]} \in \cP(D_E([0,T]))$ defined by
\begin{equation*}
\frac{\dd \bQ^g_{[0,T]}}{\dd \PR_{[0,T]}}(X) = \exp\left\{g(X(T)) - g(X(0)) - \int_0^T Hg(X(s)) \dd s \right\},
\end{equation*}
where $\PR_{[0,T]}$ is the measure $\PR$ restricted to $D_E([0,T])$. Then, the coordinate process $X$ is Markov under $\bQ^g_{[0,T]}$ and for every $f \in D$ the process $\{M^f_t\}_{t \in [0,T]}$ defined by
\begin{equation} \label{eqn:martingale_problem_of_tilted_generator}
M^f_t := f(X(t)) - f(X(0)) - \int_0^t A^g f(X(s)) \dd s
\end{equation}
is a mean $0$ martingale with respect to the filtration $\{\cF_t\}_{t \in [0,T]}$ under $\bQ^g_{[0,T]}$.
\end{proposition}

\begin{proof}
The conditions of Theorem 4.2 in \cite{PR02} are satisfied by Condition \ref{condition:D_algebra_closed_smooth} if we take for the domain of $A$ and $A^g$ the core $D$. 
\end{proof}

The transforms introduced above yield absolutely continuous trajectories of measures.

\begin{lemma} \label{lemma:trajectories_Ag_are_absolutely_continuous}
Suppose that $g \in D$ and consider the measure $\bQ^g_{[0,T]} \in \cP(D_E([0,T]))$ introduced in Proposition \ref{proposition:Girsanov_transform}. Denote by $\gamma^g(t)$ the law of $X(t)$ under $\bQ^g_{[0,T]}$. The trajectory $\{\gamma^g(t)\}_{t \in [0,T]}$ is absolutely continuous in the sense of Definition \ref{definition:absolutely_continuous} and $\dot{\gamma}^g(t) = (A^g)'(\gamma^g(t))$ for all $t \in [0,T]$.
\end{lemma}

\begin{proof}
Denote by $\{S^g(t)\}_{t \in [0,T]}$ the semigroup of conditional expectation under $\bQ^g_{[0,T]}$. Let $f \in D$ and let $\gamma(t)$ be the law of $X(t)$ under $\bQ^g_{[0,T]}$. It is straightforward to show that $S^g(t)$ is a strongly continuous semigroup on $(C(E),\vn{\cdot})$ and that $t \mapsto \gamma(t) = (S^g(t))'(\gamma(0))$ is weakly continuous. By Lemma \ref{lemma:continuity_of_operators}, we know that $A^g : (D,\tau_D) \rightarrow (C(E),\vn{\cdot})$ is continuous. We conclude that $t \mapsto (A^g)'(\gamma(t))$ is continuous in $(D,\tau_D)$. 

For every $f \in D$ and $t \geq 0$, we know that 
\begin{equation*}
\int_0^t |\ip{f}{(A^g)'(\gamma(s))}| \dd s = \int_0^t |\ip{A^g f}{\gamma(s)}| \dd s \leq t\vn{A^g f} < \infty.
\end{equation*}
Taking expectation in \eqref{eqn:martingale_problem_of_tilted_generator}, we have that
\begin{equation*}
\ip{S^g(t)f}{\gamma(0)} - \ip{f}{\gamma(0)} = \int_0^t \ip{S^g(s) A^g f}{\gamma(0)} \dd s,
\end{equation*}
which shows that
\begin{equation*}
\ip{f}{\gamma(t)} - \ip{f}{\gamma(0)} = \int_0^t \ip{A^g f}{\gamma(s)} \dd s = \int_0^t \ip{f}{(A^g)'(\gamma(s))} \dd s.
\end{equation*}
We conclude that $\{\gamma^g(t)\}_{t \in [0,T]}$ is absolutely continuous and that $\dot{\gamma}^g(t) = (A^g)'(\gamma^g(t))$ for all $t \in [0,T]$.
\end{proof}

Just as the conditional rate function $I_t$ is related to the semigroup $V(t)$, \cite{Sh85} observed for diffusion and jump-processes that the operator $L$ is related to $H$ and $A^g$ via operator duality. See also the discussion in Section 8.6.1 in \cite{FK06}.

\begin{lemma} \label{Lem:varsforH1}
Under Condition \ref{condition:D_algebra_closed_smooth}, we have for $f \in D$ that
\begin{equation}
\ip{Hf}{\mu} = \sup_{g \in D} \left\{\ip{A^g f}{\mu} - \ip{Lg}{\mu} \right\}, \label{eqn:lemvarH2}
\end{equation}
and equality holds for $g = f$. Furthermore, for $g \in D$ and $\mu \in \mathcal{P}(E)$ it holds that
\begin{equation} \label{eq:varL1}
\ip{Lg}{\mu} = \sup_{f \in D} \left\{ \ip{A^g f}{\mu} - \ip{Hf}{\mu} \right\},
\end{equation}
with equality for $f = g$.
\end{lemma}

\begin{proof}
For $\lambda > 0$,  consider the resolvent of $A$ defined by $J(\lambda)f := (\bONE - \lambda A)^{-1}f = \int_0^\infty \lambda^{-1}e^{-\lambda^{-1} t} S(t)f \dd t$. The Yosida approximants of $A$ are defined as $A_\lambda := \lambda^{-1} (J(\lambda) - \bONE) = A J(\lambda)$. It is well known that the $A_\lambda$ are bounded and are given by
\begin{equation*}
A_\lambda f(x) = \lambda^{-1} \int q_\lambda(x,\dd y) \left[f(y) - f(x)\right],
\end{equation*}
where $q_\lambda(x,\cdot)$ is the law of the process generated by $A$ after an exponential random time with mean $\lambda$. Next, define $H_\lambda, A^g_\lambda$ and $L_\lambda$ in terms of $A_\lambda$. As $A_\lambda$ is bounded, it follows by Lemma 5.7 in \cite{FK06} that
\begin{equation*}
H_\lambda f (x) \geq A_\lambda^g f(x) - L_\lambda g(x), \qquad H_\lambda f (x) = A_\lambda^f f(x) - L_\lambda f(x). 
\end{equation*}
Therefore, it follows by Yosida approximation, sending $\lambda \downarrow 0$, cf. \cite[Lemma 1.2.4]{EK86}, that $Hf(x) = \sup_{g \in D} \left\{A^g f(x) - Lg(x) \right\}$. The first statement now follows by integration. The variational statement for $L$ is obtained similarly. 
\end{proof}

\subsection{The Lagrangian and a variational expression for the Hamiltonian} \label{section:Lagrangiandef}

The Lagrangian in the previous section is still an operator acting on functions. Here we embed this object in a new Lagrangian $\cL$ that is a function of place and speed. Also, we introduce a map $\rho$ that transforms `momentum' into speed.

\begin{definition} \label{Def:Lagrangian}
Let $(D,\tau_D)$ satisfy Condition \ref{condition:topology_on_D}. Define the Lagrangian $\cL :  \mathcal{P}(E) \times D' \rightarrow [0,\infty]$ by $\cL(\mu,u) = \sup_{f \in D} \left\{ \ip{f}{u} - \ip{Hf}{\mu} \right\}$. Also, define the map $\rho : \mathcal{P}(E)\times D  \rightarrow D'$ by $\rho(\mu,g) = (A^g)'(\mu)$.
\end{definition}

Note that $\rho$ is well defined by Lemma \ref{lemma:continuity_of_operators}. $\cL$ can be considered as an extension of $L$. Pick $\mu \in \mathcal{P}(E)$ and $g \in D$, then
\begin{equation}\label{eqn:representationL}
\begin{aligned} 
\cL(\mu,\rho(\mu,g)) & = \sup_{f \in D} \left\{\ip{f}{\rho(\mu,g)} - \ip{Hf}{\mu} \right\} \\
& =  \sup_{f \in D} \left\{ \ip{A^gf}{\mu} - \ip{Hf}{\mu} \right\} = \ip{Lg}{\mu},
\end{aligned}
\end{equation}
where the last equality follows by Equation \eqref{eq:varL1}. The following result is immediate.

\begin{lemma}
$(\mu,u) \mapsto \cL(\mu,u)$ is convex and lower semi-continuous with respect to the weak and weak* topologies. 
\end{lemma}

It turns out that the space $D'$ is to large for practical purposes. In particular, it is not immediately clear that $D'$ with the weak topology is separable. In the proof of Proposition \ref{prop:compacttimelevelsets} below, we need to integrate over $D'$ and because we want to employ an extended version of the Prohorov theorem that needs separability, we will construct a more regular subspace of $D'$ that contains all relevant `speeds'.

Recall the set $\cN$ introduced in Condition \ref{condition:topology_on_D} (e) and the definition of a Polar in \eqref{eqn:polar_def}. Define $U \subseteq D'$ by $U := \bigcup_{n \in \bN} n \cN^\circ$.

We equip $U$ with the weak* topology inherited from $D'$. The importance of $U$ follows from the following lemma, which shows that we can restrict the set of allowed `speeds' to $U$.

\begin{lemma} \label{lemma:Lagrangian_infinite_out_U}
Let $\mu \in \cP(E)$. If $u \notin U$, then $\cL(\mu,u) = \infty$. Furthermore, for $\mu \in \cP(E)$ and $g \in D$, we have $\rho(\mu,g) \in U$.
\end{lemma}

\begin{proof}
For $u \notin U = \bigcup_{n} n\cN^\circ$, we can find functions $f_n \in \cN$, such that $|\ip{f_n}{u}| \geq n$. The inequality  $|\ip{f_n}{u}| \leq \cL(\mu,u) + \ip{Hf_n}{\mu} \vee \ip{H(-f_n)}{\mu}$, yields that $\cL(\mu,u) \geq n-1$ for every $n$, which implies that $\cL(\mu,u) = \infty$.

The second statement follows from the first, Equation \eqref{eqn:representationL}, and the fact that $Lg$ is bounded.
\end{proof}

As can be seen from Equation \eqref{eqn:representationL}, $\cL$ is an extension of $L$. As expected, $H$ can also be obtained by a Fenchel-Legendre transform of $\cL$.

\begin{lemma} \label{lemma:variationalexpressionH}
The variational expression for $H$ in Equation \eqref{eqn:lemvarH2} extends to $\ip{Hf}{\mu} = \sup_{u \in D'} \left\{ \ip{f}{u} - \cL(\mu,u) \right\} = \sup_{u \in U} \left\{ \ip{f}{u} - \cL(\mu,u) \right\}$.
\end{lemma}

\begin{proof}
As $\cL(\mu,u) = \infty$ if $u \notin U$, the second inequality is immediate. To prove the first equality, first note that by Definition \ref{Def:Lagrangian} of $\cL$, we have for every $f \in D$, $\mu \in \mathcal{P}(E)$, $u \in D'$ that $\ip{Hf}{\mu} \geq  \ip{f}{u} - \cL(\mu,u)$.

We now show that we in fact have equality. By Equation \eqref{eqn:representationL}, we know that $\cL(\mu,\rho(\mu,g)) = \ip{Lg}{\mu}$. Hence, by the second item in Lemma \ref{Lem:varsforH1}, we obtain
\begin{equation} \label{eq:equalityH}
\ip{Hf}{\mu} = \ip{A^f f}{\mu} - \ip{Lf}{\mu}  = \ip{f}{\rho(\mu,f)} - \cL(\mu,\rho(\mu,f)),
\end{equation}
which concludes the proof.
\end{proof}

The identification of the optimizer in the proof of Lemma \ref{lemma:variationalexpressionH} can be used to restrict to even a smaller subset of $D'$. We state the result without proof, as it will not be needed later on.

\begin{proposition} \label{prop:Lagrangian_is_infinity}
Let $\mu \in \cP(E)$ and define $\Gamma_\mu$ to be the weak* closed convex hull of $\{\rho(\mu,g) \in U \, | \,  g \in D\}$.
If $u \notin \Gamma_\mu$, then $\cL(\mu,u) = \infty$.
\end{proposition}

\section{Proof of Theorem \ref{The:LDP2}} \label{section:proof_of_ldp2}

We proceed with the proof of Theorem \ref{The:LDP2}. We start with two crucial compactness results which are necessary for the Nisio semigroup, introduced in Section \ref{section:Nisio_semigroup}, to be well behaved.

\subsection{Compactness of the space of paths with bounded Lagrangian cost}

We start with proving the compactness of the level sets of $\cL$.

\begin{proposition} \label{Prop:FKcompactness}
For each $C \geq 0$, the set
\begin{equation*}
\left\{(\mu,u) \in \cP(E) \times U \, \middle| \, \cL(\mu,u) \leq C\right\}
\end{equation*}
is compact with respect to the weak topology on $\cP(E)$ and the weak* topology on $U$.
\end{proposition}

\begin{proof}
First of all, as $\cL$ is lower semi-continuous $\{(\nu,u) \in \cP(E) \times U \, | \, \cL(\nu,u) \leq C\}$ is closed. We show that it is contained in a compact set.

Pick the neighbourhood $\cN$ of $0$ that was given in Condition \ref{condition:topology_on_D} (e). Recall that in a barrelled space every barrel is a neighbourhood of $0$. Set $M := \sup_{f \in \cN} \vn{Hf}$. As $\ip{f}{u} \leq \cL(\mu,u) + \ip{Hf}{\mu}$, we obtain
\begin{equation*}
|\ip{f}{u}| \leq \cL(\nu,u) + \ip{Hf}{\nu} \vee \ip{H(-f)}{\nu}.
\end{equation*}
As a consequence,
\begin{equation*}
\left\{(\nu,u) \in \cP(E) \times U \, \middle| \, \cL(\nu,u) \leq C\right\} \subseteq \cP(E) \times |C + M|\cN^\circ.
\end{equation*}
Because $(D',wk^*)$ is Hausdorff and a locally convex space, the closure of this set is compact in $(D',wk^*)$ by the Bourbaki-Aloaglu theorem\cite[Propositions 32.7 and 32.8]{Tr67}, \cite[Theorem III.6]{RR73}.
\end{proof}

We now state an essential ingredient of the proof of Theorem \ref{The:LDP2}.

\begin{proposition} \label{prop:compacttimelevelsets}
For each $M > 0$, and time $T \geq 0$,
\begin{equation*}
\cK_{M}^T := \left\{\nu \in C_{\cP(E)}([0,T]) \, \middle| \,  \nu \in \cA\cC, \int_0^T \cL(\nu(s),\dot{\nu}(s)) \dd s \leq M \right\}
\end{equation*}
is a compact subset of $C_{\cP(E)}([0,T])$.
\end{proposition}

We postpone the lengthy proof of this proposition to Sections \ref{section:preparations_proof_compact_level_sets} and \ref{section:proof_compact_level_sets}. Using the techniques introduced in these sections, we will prove Proposition \ref{proposition:finite_lagrangian_cost_implies_strong_abs_cont} in Section \ref{section:proof_abs_cont}. We focus on proving Theorem \ref{The:LDP2} first, which is done in Sections \ref{section:Nisio_semigroup} to \ref{section:Lagrangian_form}. 

\smallskip

These sections are organised as follows. In Section \ref{section:Nisio_semigroup}, we introduce the Nisio semigroup and prove some basic properties of this semigroup. In Sections \ref{section:first_inequality} and \ref{section:second_inequality}, we prove that the Nisio semigroup bounds the lift of the non-linear semigroup $V(t)$ to $\cP(E)$ from below and from above. In Section \ref{section:Lagrangian_form}, we show that the equality of the two semigroups leads to a Lagrangian form of the rate function.

\subsection{The Nisio semigroup} \label{section:Nisio_semigroup}

\begin{definition} 
The Nisio semigroup $\mathbf{V}$ mapping upper semi-continuous functions on $\cP(E)$ to upper semi-continuous functions on $\cP(E)$ is defined by
\begin{equation*}
\mathbf{V}(t) G(\mu) = \sup_{\nu \in \cA \cC_\mu} \left\{G(\nu(t)) - \int_0^t \cL(\nu(s),\dot{\nu}(s)) \dd s \right\}.
\end{equation*}
\end{definition} 

For a function $f \in C(E)$, we denote with $[f]$ the weakly continuous function on $\cP(E)$ defined by $[f](\mu) = \ip{f}{\mu}$.
Our goal in the next three sections is to show that $\mathbf{V}(t)[f](\mu) = \ip{V(t)f}{\mu}$.

\smallskip

Note that as a direct consequence of Proposition \ref{prop:compacttimelevelsets}, if $G$ is a bounded continuous function, than the supremum is actually attained by a curve starting at $\mu$ in $\cK_{3 \vn{G}}^t$. For example, this is the case if $G = [g]$, for $g \in C(E)$.

\smallskip

We need one small result which states that for sufficiently many $f$, there exists a path such that equality is attained in Young's inequality for every time $t$. The Lemma is used for the proof of Lemma \ref{lemma:Resolvents_inequality} below. 

\begin{lemma} \label{lemma:con810811inFK}
For each $\mu \in \mathcal{P}(E)$ and $f \in D$, there exists $\nu \in \cA\cC_\mu$ such that for every $t \geq 0$
\begin{equation*}
\int_0^t \ip{f}{\dot{\nu}(s)} \dd s = \int_0^t  \ip{Hf}{\nu(s)} + \cL(\nu(s),\dot{\nu}(s)) \dd s.
\end{equation*}
\end{lemma}

In particular by taking $f = 0$, we find that there is a path with zero cost. This in turn yields $\mathbf{V}(t)\mathbf{0} = \mathbf{0}$, where $\mathbf{0}$ is the function defined by $\mathbf{0}(\mu) = 0$ for all $\mu \in \cP(E)$.

\begin{proof}
Let $\nu(s)$ be the path obtained by the time projections of the Markov process started at $\mu$ generated by the operator $A^f$, see Proposition \ref{proposition:Girsanov_transform}. This gives us a path such that $\dot{\nu}(s) = (A^f)'(\nu(s)) = \rho(\nu(s),f)$.

By Equation \eqref{eq:equalityH} on page \pageref{eq:equalityH}, it follows that
\begin{equation*}
\ip{Hf}{\nu(s)} = \ip{f}{\rho(\nu(s),f)} - \cL(\nu(s),\rho(\nu(s),f))
\end{equation*}
for every $s$, implying $\int_0^t \ip{Hf}{\nu(s)} \dd s = \int_0^t \left( \ip{f}{\dot{\nu}(s)} - \cL(\nu(s),\dot{\nu}(s)) \right)  \dd s$.
\end{proof}

Just like the semigroup $\{V(t)\}_{t \geq 0}$, the Nisio semigroup $\{\mathbf{V}(t)\}_{t \geq 0}$ enjoys good continuity properties. 

\begin{lemma} \label{lemma:Nisio_V_contractive}
For every $t \geq 0$, $\mathbf{V}(t)$ is contractive, i.e. for bounded and upper semi-continuous functions $F,G$, we have $\vn{\mathbf{V}(t)F - \mathbf{V}(t)G} \leq \vn{F-G}$.
\end{lemma}

The proof of this lemma is straightforward. The next result can be proven using Proposition \ref{prop:compacttimelevelsets} as Lemma 8.16 in \cite{FK06}.

\begin{lemma} \label{lemma:Nisio_cont_in_t}
For every $f \in C(E)$ and $\mu \in \cP(E)$, we have that $t \mapsto \mathbf{V}(t)[f](\mu)$ is continuous.
\end{lemma}

Now that the basic properties of the Nisio-semigroup are known, we proceed with the proof that $V(t)f = \mathbf{V}(t)[f]$. The argument is split into two steps. The inequality $\ip{V(t)f}{\mu} \geq \mathbf{V}(t)[f](\mu)$ is established by arguments based on approximation of the semigroups by their resolvents. The second inequality is based on a Doob-h transform argument.

\subsection{The first inequality between the two semigroups} \label{section:first_inequality}

For $f \in C(E)$, define $J(\lambda)f := (\bONE - \lambda A)^{-1}f = \int_0^\infty \lambda^{-1}e^{-\lambda^{-1} t} S(t)f \dd t$. Using $J(\lambda)$, we set $R(\lambda)f := \log J(\lambda)e^f$.

\smallskip

We constructed the semigroup $V(t)$ from the linear semigroup $S(t)$, and the operator $R(\lambda)$ from the linear resolvent $J(\lambda)$. One would therefore hope that $R(\lambda)$ equals $(\bONE - \lambda H)^{-1}$. This is not the case, but we do have the following two results, which we will need for the proof of Lemma \ref{lemma:Resolvents_inequality} and Proposition \ref{proposition:VisV1}.

\begin{lemma} \label{lemma:Resolvent_almost}
For $f \in C(E)$ and $\lambda >0$, we have $R(\lambda)f \in \cD(H)$ and $(\bONE - \lambda H)R(\lambda)f \geq f$.
\end{lemma}

\begin{proof}
$J(\lambda)$ maps $C(E)$ bijectively on $\cD(A)$, therefore, $e^{R(\lambda)f} = J(\lambda)e^f \in \cD(A)$. Thus by Corollary \ref{corollary:calculationH}, we have that $R(\lambda)f \in \cD(H)$.

\smallskip

Let $x \in E$, we prove $\left(\bONE - \lambda H\right)R(\lambda) f(x) \geq f(x)$. First, we show that 

We prove that the following quantity is larger than $0$:
\begin{align*}
\left(\bONE - \lambda H\right)R(\lambda) f(x) - f(x) & = R(\lambda)f(x) - f(x) - \lambda \frac{A J(\lambda)e^f(x)}{J(\lambda)e^f(x)} \\
& = R(\lambda)f(x) - f(x) - \frac{J(\lambda) e^f(x) - e^{f(x)}}{J(\lambda) e^f(x)}.
\end{align*}
This is equivalent to showing that
\begin{align*}
J(\lambda) e^f(x) \log\left(J(\lambda)e^f(x)\right) - f(x) J(\lambda)e^f(x) - J(\lambda) e^f(x) + e^{f(x)} 
\end{align*}
is positive, which follows from the fact that for every $c \in \bR$, the function defined for non-negative $y$, given by $y \mapsto y \log y - (c+1)y + e^c$, is non-negative.
\end{proof}

Note that the fact that the function $y \mapsto y \log y - (c+1)y + e^c$ has a unique point where it hits $0$. This means that $(\bONE - \lambda H)R(\lambda)f(x) = f(x)$ only if $\bE[e^{f(X_\tau)} \, | \, X_0 = x] = e^{f(x)}$, where $\tau$ is an exponential random variable with mean $\lambda$ independent of the process $X$. This can not be true in general.

\smallskip

Even though $R(\lambda)$ does not invert $(\bONE - \lambda H)$, it does approximate the semigroup in a way that the resolvents of $H$ would as well.

\begin{lemma} \label{lemma:R_to_V}
For every $f \in C(E)$, we have that $\lim_{n\rightarrow \infty} R\left(n^{-1}\right)^{\lfloor nt \rfloor} f = V(t)f$.
\end{lemma}

\begin{proof}
By definition, we have $R\left(n^{-1}\right)^{\lfloor nt \rfloor} f = \log J\left(n^{-1}\right)^{\lfloor nt \rfloor} e^f$. For linear semigroups, we know that the resolvents approximate the semigroup: $J\left(\frac{1}{n}\right)^{\lfloor nt \rfloor} e^f \rightarrow S(t)e^f$, see for example Corollary 1.6.8 in \cite{EK86}. Therefore, by uniform continuity of the logarithm on $[e^{-\vn{f}},e^{\vn{f}}]$, we obtain the  final result by applying the logarithm.
\end{proof}

In the next definition, we introduce the resolvent $\mathbf{R}(\lambda)$ of the Nisio semigroup. Using Lemma \ref{lemma:Resolvent_almost}, we show that $\mathbf{R}(\lambda)[f](\mu) \leq [R(\lambda)f](\mu)$ which by approximation yields $\mathbf{V}(t)[f](\mu) \leq \ip{V(t)f}{\mu}$. 

\begin{definition}
Let $G$ be upper semi-continuous and bounded and let $\lambda > 0$. Define the resolvent $\mathbf{R}(\lambda)$ by
\begin{equation*}
\mathbf{R}(\lambda)G(\mu) = \sup_{\nu \in \cA\cC_\mu} \int_0^\infty \frac{1}{\lambda} e^{-\lambda^{-1} s} \left[G(\nu(s)) - \int_0^s \cL(\nu(r),\dot{\nu}(r)) \dd r \right]\dd s.
\end{equation*}
\end{definition}

\begin{lemma} \label{lemma:Resolvents_inequality} 
For $g \in D$ and $\lambda >0$, we have $\mathbf{R}(\lambda)[(\bONE - \lambda H)g] = [g]$. As a consequence, we have for $f \in C(E)$ and $\mu \in \cP(E)$ that
\begin{equation} \label{eqn:equality_of_resolvents}
\mathbf{R}(\lambda)[f](\mu) \leq [R(\lambda)f](\mu).
\end{equation}
\end{lemma} 

\begin{proof}
The first statement follows along the lines of the proof of Lemma 8.19 in \cite{FK06}. Summarising, the inequality $\mathbf{R}(\lambda)[(\bONE - \lambda H)g] \leq [g]$
follows by integration by parts and Young's inequality:
\begin{equation*} 
\ip{g}{u} \leq \ip{Hg}{\mu} + \cL(\mu,u), \quad \mu \in \cP(E), \, u \in D, \, g \in C(E).
\end{equation*}
The second inequality, $\mathbf{R}(\lambda)[(\bONE - \lambda H)g] \geq [g]$, follows by integration by parts and Lemma \ref{lemma:con810811inFK}, which gives us a trajectory for which equality is attained for all times in Young's inequality.

For the second statement, first note that if $F \geq G$, then $\mathbf{R}(\lambda) F \geq \mathbf{R}(\lambda)G$. Therefore, we obtain by Lemma \ref{lemma:Resolvent_almost} that
\begin{equation*}
\mathbf{R}(\lambda)[f](\mu) \leq \mathbf{R}(\lambda)[(\bONE -\lambda H) R(\lambda)f](\mu) = \ip{R(\lambda)f}{\mu}.
\end{equation*}
\end{proof}

The next lemma relies on Lemma \ref{lemma:Nisio_cont_in_t} and follows exactly as Lemma 8.18 in \cite{FK06}.

\begin{lemma}\label{lemma:bR_to_bV}
For $t \geq 0$, $f \in D$ and $\mu \in \cP(E)$, we have
\begin{equation*}
\lim_{n \rightarrow \infty}\mathbf{R}(n)^{\lfloor nt \rfloor}[f](\mu) = \mathbf{V}(t)[f](\mu).
\end{equation*}
\end{lemma}

We are now able to prove the first inequality between the Nisio semigroup $\{\mathbf{V}(t)\}_{t \geq 0}$ and $\{V(t)\}_{t \geq 0}$.

\begin{proposition} \label{proposition:VisV1}
For $t \geq 0$, $f \in C(E)$ and $\mu \in \cP(E)$, we have
\begin{equation*}
\mathbf{V}(t)[f](\mu) \leq \ip{V(t)f}{\mu}.
\end{equation*}
\end{proposition}

\begin{proof}
By repeatedly using Equation \eqref{eqn:equality_of_resolvents}, we obtain 
\begin{equation*}
\mathbf{R}(n^{-1})^{\lfloor nt\rfloor}[f](\mu) \leq \ip{R(n^{-1})^{\lfloor nt\rfloor}f}{\mu},
\end{equation*}
which implies by Lemmas \ref{lemma:R_to_V} and \ref{lemma:bR_to_bV} that $\mathbf{V}(t)[f](\mu) \leq \ip{V(t)f}{\mu}$.
\end{proof}

\subsection{The second inequality between the two semigroups} \label{section:second_inequality}

The second inequality $\mathbf{V}(t)[f](\mu) \geq \ip{V(t)f}{\mu}$ needs more work. As in the proof of Lemma 4.10 in \cite{DG87}, we will argue via the Doob-h transform.  We have the following useful variant of Lemma 2.19 in \cite{Se93}. 

\begin{lemma} \label{lemma:lvl2_Vt_in_terms_of_restricted_entropy}
Let $\PR \in \cP(D_E(\bR^+))$ be Markov with transition semigroup $\{S(t)\}_{t \geq 0}$. Let $h \in C(E)$ and let $t > 0$. Set 
\begin{equation*}
S(\bQ) = \begin{cases}
H(\bQ \, | \, \PR) & \text{if } \bQ_0 = \PR_0, \\
\infty & \text{otherwise}.
\end{cases}
\end{equation*}
Then,
\begin{equation*}
\ip{V(t)h}{\PR_0} = \sup_{\bQ \in \cP(D_E(\bR^+))} \left\{ \ip{h}{\bQ_t} - S(\bQ) \right\},
\end{equation*}
where $\bQ_t$ denotes the time $t$ marginal of $\bQ$. The supremum is attained by the measure $\bQ^h$ defined by
\begin{equation*}
\frac{\dd \bQ^h}{\dd \PR}(X) = \frac{e^{h(X(t))}}{\ip{e^h}{\PR_t}} = e^{h(X(t)) - \ip{V(t)h}{\PR_0}}.
\end{equation*}
\end{lemma}

\begin{proof}
Let $\PR_{0,t} \in \cP(E^2)$ be the restriction of $\PR$ to the time $0$ and time $t$ marginals. As before, we denote by $\PR_0$ the time $0$ marginal of $\PR$ and for a measure $\nu \in \cP(E^2)$ we denote by $\nu_0$ respectively $\nu_1$ the restriction to the first marginal and second marginal. Set
\begin{equation*}
S_t(\nu) = \begin{cases}
H(\nu \, | \, \PR_{0,t}) & \text{if } \nu_0 = \PR_0, \\
\infty & \text{otherwise}.
\end{cases}
\end{equation*}
By Lemma 2.19 in \cite{Se93} and convex duality, we obtain
\begin{equation*}
\ip{V(t)h}{\PR_0} = \sup_{\nu \in \cP(E^2)} \left\{\ip{h}{\nu_2} - S_t(\nu) \right\}.
\end{equation*}
By the contraction principle, we have
\begin{equation*}
H(\nu \, | \, \PR_{0,t})  = \inf \left\{H(\bQ \, | \, \PR) \, | \, \bQ \in \cP(D_E(\bR^+)) \, : \, \bQ_{0,t} = \nu \right\},
\end{equation*}
which implies that
\begin{equation*}
\ip{V(t)h}{\mu} = \sup_{\bQ \in \cP(D_E(\bR^+))} \left\{\ip{h}{\bQ_t} - S(\bQ) \right\}.
\end{equation*}
Now we show that the supremum is achieved for $\bQ^h$ defined by
\begin{equation*}
\frac{\dd \bQ^h}{\dd \PR}(X) = \frac{e^{h(X(t))}}{\ip{e^{h}}{\PR_t}} = e^{h(X(t)) - \ip{V(t)h}{\PR_0}}.
\end{equation*}
Note that $\bQ^h_0 = \PR_0$. Therefore, we obtain that
\begin{multline*}
\ip{h}{\bQ^h_t} - S(\bQ^h) = \ip{h}{\bQ^h_t} - \int \log \frac{\dd \bQ^h}{\dd \PR} \dd \bQ^h \\
= \ip{h}{\bQ^h_t} - \ip{h}{\bQ^h_t} + \ip{V(t)h}{\PR_0} = \ip{V(t)h}{\PR_0}.
\end{multline*} 
\end{proof}

The optimising measure $\bQ^h$ defined in the lemma above has the form of a Doob-h transform, see Doob \cite[page 566]{Do84} or \cite{Ja75,FG97}. For $s \leq t$, define $h(s) = V(t-s) h$, or $e^{h(s)} = S(t-s) e^h$. 

The transition probabilities of the Markov process described by $\bQ^h$ up to time $t$ can be written down explicitly as a semigroup of transition operators $\left\{S^{h[0,t]}(r,s)\right\}_{0\leq r\leq s \leq t}$, where $S^{h[0,t]}(r,s) : C(E) \rightarrow C(E)$ is defined by $S^{h[0,t]}(r,s)f(x) := \bQ^h[f(X(s)) \, | \, X(r) = x]$.  The following result is obtained by a straightforward calculation.

\begin{lemma} \label{lemma:lvl2_transition_semigroup_of_Doob_transform}
The semigroup of transition probabilities of $\bQ^h$ defined by
\begin{equation*}
\frac{\dd \bQ^h}{\dd \PR}(X) = \frac{e^{h(X(t))}}{\ip{e^h}{\PR_t}} = e^{h(X(t)) - \ip{V(t)h}{\PR_0}},
\end{equation*}
is given by
\begin{equation*}
S^{h[0,t]}(r,s)f(x) = e^{-h(r)}(x)S(s-r)\left(fe^{h(s)}\right)(x).
\end{equation*}
\end{lemma}

Because $h \in D$, we find $e^f \in D \subseteq \cD(A)$. As $\cD(A)$ is preserved under the semigroup $\{S(t)\}_{t \geq 0}$, we find $e^{h(s)} \in \cD(A)$ and $h(s) \in \cD(H)$. By Corollary \ref{corollary:calculationH}, we have $\frac{\dd}{\dd s} h(s) = - Hh(s)$. We conclude that
\begin{multline} \label{eqn:rewrite_H_doob_transform}
h(t)X(t) - h(0)(X(0)) \\
= h(t)X(t) - h(0)(X(0)) - \int_0^t H h(s)(X(s)) + \frac{\dd}{\dd s} h(s) (X(s)) \dd s.
\end{multline}

Suppose that $s \mapsto h(s)$ would be continuous in $(D,\tau_D)$, then it is possible to prove a time-dependent version of Proposition \ref{proposition:Girsanov_transform}. This would give that the process $\{X(s)\}_{s \leq t}$ is Markovian under $\bQ^h$ with time dependent generator $s \mapsto A^{h(s)}$ and that for $s \leq t$
\begin{equation*}
M^h_s := h(s)(X(s)) - h(0)(X(0)) - \int_0^s A^{h(s)} h(s)(X(s)) + \frac{\dd}{\dd s}h(s)(X(s)) \dd s
\end{equation*}
is a martingale under $\bQ^h$. Thus,
\begin{align*}
S(\bQ^h) & = H(\bQ^h \, | \, \PR) = \int \log \frac{\dd \bQ^h}{\PR} \dd \bQ^h \\
& = \int  h(t)X(t) - h(0)(X(0)) \\
& \qquad \qquad - \int_0^t H h(s)(X(s)) + \frac{\dd}{\dd s} h(s) (X(s)) \, \dd s \, \bQ^h(\dd X) \\
& = \int M_t^h + \int_0^t A^h(s) h(s) (X(s)) - H(s)(X(s)) \, \dd s \, \bQ^h(\dd X) \\
& = \int \int_0^t Lh(s) (X(s)) \,  \dd s \, \bQ^h(\dd X).
\end{align*} 
Formally interchanging the two integrals yields
\begin{equation} \label{eqn:Doob_first_try_integrals_interchange}
S(\bQ^h) = \int_0^t \int  Lh(s) (X(s)) \,  \bQ^h(\dd X) \, \dd s = \int_0 ^t \cL(\gamma(s),\dot{\gamma}(s)) \dd s,
\end{equation}
where $\gamma(s)$ is the law of $X(s)$ under $\bQ^h$. This would yield the inequality $\mathbf{V}(t)[h](\mu) \geq \ip{V(t)h}{\mu}$, for the Markov process $X$ with starting law $\mu$. However, we made two assumptions that are not necessarily satisfied, so we need to refine our argument.

\smallskip

In \cite{DG87} a similar issue plays a role in the proof of Lemma 4.10. In this context, diffusion processes are considered and $D$ is the space of compactly supported smooth functions. This space is not closed under the evolution of the semigroup $\{S(t)\}_{t \geq 0}$, so also there a refined argument is used. To be precise, the law of the Doob-transform is approximated by the law of Doob-transforms of processes that are killed upon leaving a ball with large radius. It is shown that the law of the Doob-transform is sufficiently well approximated by the killed processes, to still conclude the desired inequality.

\smallskip

Here, we will also consider an approximation of the Doob-transform. Based on the discussion above, we known that the Markov process obtained via the Doob-transform formally has generator $A^{h(s)}$ at time $s$. Below, we will approximate the function $s \mapsto h(s)$ by a collection of piecewise-constant maps $s \mapsto g_n(s) := \sum_{i=1}^n g_{n,i} \bONE_{\{n^{-1} t(i-1) < s \leq n^{-1} t i\}}$ taking its values in $D$. The Markov process with time-dependent generator $s \mapsto A^{g_n(s)}$ can be obtained from $\PR$ via a Girsanov transform and we show that this process converges to the Doob-transform $\bQ^h$ in entropy.

Recall that the logarithm of the change of measures $\frac{\dd \bQ^h}{\dd \PR}$ was given by
\begin{multline*} 
h(t)X(t) - h(0)(X(0)) \\
= h(t)X(t) - h(0)(X(0)) - \int_0^t H h(s)(X(s)) + \frac{\dd}{\dd s} h(s) (X(s)) \dd s,
\end{multline*}
and that $e^{h(s)} = S(t-s) e^h$. So in particular, as $h \in D$, we find $e^f \in D \subseteq \cD(A)$. As $\cD(A)$ is preserved under the semigroup, we find $e^{h(s)} \in \cD(A)$ and by Corollary \ref{corollary:calculationH} that $h(s) \in \cD(H)$.

Both $s \mapsto A e^{h(s)}$ and $s \mapsto e^{h(s)}$ are norm continuous. As $\inf_x \inf_s e^{h(s)}(x) > 0$, we find that $s \mapsto h(s)$ and $s \mapsto e^{-h(s)}$, and as a consequence $s \mapsto Hh(s)= - \frac{\dd}{\dd s} h(s)$ are also norm continuous. Thus, for a fixed $\varepsilon > 0$ we can choose $N$ such that for $n\geq N$: 
\begin{enumerate}[(a)]
\item we have
\begin{equation} \label{eqn:Doob_approx_1}
\sup_{1 \leq i \leq n} \quad \sup_{t(i-1) \leq s \leq ti/n} t\vn{H h(s) - Hh\left(\frac{it}{n}\right)} + \vn{ h(s) -h\left(\frac{it}{n}\right)} \leq \varepsilon.
\end{equation}
\item for all $i \in \{1,\dots,n\}$:
\begin{equation}\label{eqn:Doob_approx_2}
t \vn{\frac{\dd}{\dd s}h(s)|_{s = \frac{ti}{n}} - \frac{h\left(\frac{it}{n}\right) - h\left(\frac{(i-1)t}{n}\right)}{t/n}} \leq \varepsilon.
\end{equation}
\end{enumerate}

As $h(s) \in \cD(H)$, we can use Lemma \ref{lemma:approximation_core_H} to find for $n\geq N$ functions $g_{n,i} \in D$ such that
\begin{equation} \label{eqn:Doob_approx_3}
\sup_{1 \leq i \leq n} n\vn{g_{i,n} - h\left(\frac{it}{n}\right)} + t\vn{H g_{i,n} - Hh\left(\frac{it}{n}\right)} \leq \frac{\varepsilon}{n}.
\end{equation}

Consider the maps $G_n : D_E(\bR^+) \rightarrow \bR$ defined by
\begin{multline} \label{eqn:def_Gn}
G_n(X) = \sum_{i = 1}^{n} g_{n,i}\left(X\left(\frac{it}{n}\right)\right) - g_{n,i}\left(X\left(\frac{(i-1)t}{n}\right)\right) \\
- \int_{\frac{(i-1)t}{n}}^{\frac{it}{n}} H g_{n,i}(X(s)) \dd s
\end{multline}
and set $H(X) := h(t)X(t) - h(0)(X(0))$. We will show that the measures $\bQ[G_n]$ defined by
\begin{equation} \label{eqn:QGn}
\frac{\dd \bQ[G_n]}{\dd \PR}(X) = e^{G_n(X)}
\end{equation}
are approximating the measure $\bQ^h$ in entropy:
\begin{align*}
H(\bQ[G_n] \, | \, \PR) \rightarrow H(\bQ^h \, | \, \PR), \qquad H(\bQ^h \, | \, \bQ[G_n]) \rightarrow 0.
\end{align*}

We start by proving that the functions $G_n$ and $H$ are uniformly bounded.

\begin{lemma} \label{lemma:approx_Doob_transform_uniform_bound}
There is a constant $M > 0$ such that
\begin{equation*}
\sup_n \sup_{x \in D_E(\bR^+)} |G_n(x)| + |H(x)| \leq M.
\end{equation*}
\end{lemma}

\begin{proof}
It is clear that $\sup_x |H(x)| \leq \vn{h(t)} + \vn{h(0)}$. For $G_n$, we first consider the integral part. Fix some $\varepsilon > 0$ and fix $N$ such that for all $n \geq N$ \eqref{eqn:Doob_approx_3} is satisfied, then
\begin{align*}
& \sum_{i=1}^n\int_{\frac{t(i-1)}{n}}^{\frac{it}{n}} \left|H g_{n,i}(X(s))\right| \dd s \\
& \leq \sum_{i=1}^n\int_{\frac{t(i-1)}{n}}^{\frac{it}{n}} \left|H g_{n,i}(X(s)) - Hh\left(\frac{ti}{n}\right)(X(s))\right| + \left|Hh\left(\frac{ti}{n}\right)(X(s))\right| \dd s \\
& \leq \varepsilon + \sup_{s \in [0,t]} t\vn{Hh(s)}.
\end{align*}
For the remainder, we first rearrange:
\begin{multline*}
\sum_{i = 1}^{n} g_{n,i}\left(X\left(\frac{it}{n}\right)\right) - g_{n,i}\left(X\left(\frac{(i-1)t}{n}\right)\right) \\
= g_{n,n}(X(t)) - g_{n,1}(X(0)) - \sum_{i=1}^{n-1} g_{n,i+1}\left(X\left(\frac{ti}{n}\right)\right) -g_{n,i}\left(X\left(\frac{ti}{n}\right)\right).
\end{multline*}
Because $g_{n,n} \rightarrow h(t)$ and $g_{n,1} \rightarrow h(0)$, there is some $M > 0$ such that
\begin{equation*}
\sup_x |g_{n,n}(x(t))| + |g_{n,1}(x(0))| \leq M.
\end{equation*}
For the terms in the sum, we compare to the functions $\{h(s)\}_{s \in [0,t]}$:
\begin{multline*}
\vn{g_{n,i+1} - g_{n,i}} \leq \vn{g_{n,i+1} - h\left(\frac{t(i+1)}{n}\right)} \\
+ \vn{h\left(\frac{t(i+1)}{n}\right) - h\left(\frac{ti}{n}\right)} + \vn{h\left(\frac{ti}{n}\right) - g_{n,i}} \\
\leq \frac{2 \varepsilon}{n} +  \vn{h\left(\frac{t(i+1)}{n}\right) - h\left(\frac{ti}{n}\right)},
\end{multline*}
where we have used \eqref{eqn:Doob_approx_3}. The final term can be bounded using \eqref{eqn:Doob_approx_2}:
\begin{multline*}
\vn{h\left(\frac{t(i+1)}{n}\right) - h\left(\frac{ti}{n}\right)} = \frac{t}{n} \vn{\frac{h\left(\frac{it}{n}\right) - h\left(\frac{(i-1)t}{n}\right)}{t/n}} \\
\leq \frac{t}{n}\vn{\frac{\dd}{\dd s}h(s)|_{s = \frac{ti}{n}} - \frac{h\left(\frac{it}{n}\right) - h\left(\frac{(i-1)t}{n}\right)}{t/n}} + \frac{t}{n} \sup_{s \in [0,t]} \vn{\frac{\dd}{\dd s} h(s)} \\
\leq \frac{\varepsilon}{n} + \frac{t}{n} \sup_{s \in [0,t]} \vn{Hh(s)}.
\end{multline*}
We conclude that for some $N$ and all $n \geq N$, we have
\begin{equation*}
\sup_{n \geq N} \sup_{x \in D_E(\bR^+)} |G_n(x)| + |H(x)| \leq \vn{h(0)} + \vn{h(t)} + M + 2t \sup_s \vn{Hh(s)} + 4 \varepsilon,
\end{equation*}
which concludes the proof.
\end{proof}

\begin{proposition}\label{proposition:approx_of_doob_transform_analytic}
For every $\eta > 0$, there exists an integer $N \geq 1$ and a measurable set $\fS_{N,\eta} \subseteq D_E(\bR^+)$ such that for $n \geq N$
\begin{equation*}
\PR[\fS_{N,\eta}] > 1- \eta, \quad \text{ and } \quad \sup_{X \in \fS_{N,\eta}} \left|G_n(X) - H(X)\right|  \leq \eta.
\end{equation*}
\end{proposition}

We start with the probabilistic content of the proposition.

\begin{lemma} \label{lemma:approximation_of_doob_transform_good_set}
Denote
\begin{equation*}
\Upsilon_n(X,s) := \frac{\dd}{\dd s} h(s)(X(s)) - \sum_{i = 1}^n \bONE_{\{(i-1)t/n < s \leq it/n\}}\frac{\dd}{\dd s} h(s)\left(X\left(\frac{ti}{n}\right)\right),
\end{equation*}
and $\overline{\Upsilon}_N(X,s) := \sup_{n \geq N} \Upsilon_n(X,s)$. We have
\begin{equation*}
\lim_{N \rightarrow \infty} \bE\left[\int_0^t \left| \overline{\Upsilon}_N(X,s) \right| \dd s \right] = 0.
\end{equation*}
There is an $N$, such that
\begin{equation*}
\PR\left[\fS_{N,\varepsilon,\eta} \right] \geq 1- \eta, \qquad \fS_{N,\varepsilon,\eta} := \left\{\forall n \geq N: \, \int_0^t |\Upsilon_n(X,s)| \dd s \leq \varepsilon \right\}.
\end{equation*}
\end{lemma}

\begin{proof}
By the right continuity of paths in the Skorokhod space, the first claim follows by the Dominated convergence theorem. The second claim is a consequence of the first claim and Markov's inequality.
\end{proof}

\begin{proof}[Proof of Proposition \ref{proposition:approx_of_doob_transform_analytic}]
Choose $\eta > 0$ and let $\varepsilon = \frac{\eta}{10}$. Now let $N$ be large enough such that the result in Lemma \ref{lemma:approximation_of_doob_transform_good_set} holds and denote $\fS_{H,\eta} := \fS_{N,\eta/10,\eta}$. Additionally, let $N$ be large enough such that for $n \geq N$ the approximations in \eqref{eqn:Doob_approx_1}, \eqref{eqn:Doob_approx_2} and \eqref{eqn:Doob_approx_3} are valid.

Let $n \geq N$. Reordering the first sum of $G_n$ yields
\begin{align*}
G_n(X) & = g_{n,n}(X(t)) - g_{n,1}(X(0)) - \sum_{i = 1}^{n} \int_{\frac{(i-1)t}{n}}^{\frac{it}{n}} H g_{n,i}(X(s)) \dd s \\
& \qquad - \sum_{i=1}^{n-1} \int_{\frac{it}{n}}^{\frac{(i+1)t}{n}} \frac{g_{n,i+1}\left(X\left(\frac{ti}{n}\right)\right) - g_{n,i}\left(X\left(\frac{ti}{n}\right)\right)}{t/n} \dd s,
\end{align*}
whereas, by \eqref{eqn:rewrite_H_doob_transform}:
\begin{equation*}
H(X) = h(t)X(t) - h(0)(X(0)) - \int_0^t H h(s)(X(s)) + \frac{\dd}{\dd s} h(s) (X(s)) \dd s.
\end{equation*}
We compare the terms of $G$ and $H$. First of all, by \eqref{eqn:Doob_approx_3}, we have $\vn{h(t) - g_{n,n}} \leq \varepsilon$, secondly by \eqref{eqn:Doob_approx_1} and \eqref{eqn:Doob_approx_3}, we have
\begin{equation*}
\vn{h(0) - g_{n,1}} \leq \vn{h(0) - h\left(\frac{t}{n}\right)} + \vn{h\left(\frac{t}{n}\right) - g_{n,1}} \leq 2 \varepsilon.
\end{equation*}
Again by \eqref{eqn:Doob_approx_1} and \eqref{eqn:Doob_approx_3}, we have for $\frac{(i-1)t}{n} \leq s \leq \frac{it}{n}$ that
\begin{multline*}
\int_{\frac{(i-1)t}{n}}^{\frac{ti}{n}}\vn{Hh(s) - H g_{n,i}} \dd s \\
\leq \int_{\frac{(i-1)t}{n}}^{\frac{ti}{n}} \vn{Hh(s) - Hh\left(\frac{it}{n}\right)} + \vn{Hh\left(\frac{it}{n}\right) - H g_{n,i}} \dd s  \leq \frac{2 \varepsilon}{n}.
\end{multline*}
The remaining difference is given by
\begin{multline} \label{eqn:doob_approximation_final_difference}
\int_0^t \frac{\dd}{\dd s} h(s) (X(s)) \dd s - \sum_{i=1}^{n-1} \int_{\frac{it}{n}}^{\frac{(i+1)t}{n}} \frac{g_{n,i+1}\left(X\left(\frac{ti}{n}\right)\right) - g_{n,i}\left(X\left(\frac{ti}{n}\right)\right)}{t/n} \dd s.
\end{multline}
We restrict ourselves to the set $\fS_{N,\eta} = \fS_{N,\eta/10,\eta}$ defined in the second claim of Lemma \ref{lemma:approximation_of_doob_transform_good_set}. On $\fS_{N,\eta}$, we can replace $X(s)$ in the first integral by $X\left(\frac{ti}{n}\right)$ for the appropriate $i$ at the cost of an error of size $\varepsilon$. Thus, it is sufficient to give an upper bound in terms of the supremum norm.

\smallskip

For $\frac{(i-1)t}{n} \leq s \leq \frac{it}{n}$, we obtain by \eqref{eqn:Doob_approx_1}, \eqref{eqn:Doob_approx_2} and \eqref{eqn:Doob_approx_3} that
\begin{multline*}
\vn{\frac{\dd}{\dd s} h(s) - \frac{g_{n,i+1} - g_{n,i}}{t/n}}  \\
\leq \vn{\frac{\dd}{\dd s} h(s) - \frac{\dd}{\dd s} h\left(\frac{ti}{n}\right)} + \vn{\frac{\dd}{\dd s} h\left(\frac{ti}{n}\right) - \frac{h\left(\frac{t(i+1)}{n}\right) - h\left(\frac{ti}{n}\right)}{t/n}} \\
+ \vn{\frac{h\left(\frac{t(i+1)}{n}\right) - h\left(\frac{ti}{n}\right)}{t/n} - \frac{g_{n,i+1} - g_{n,i}}{t/n}} \leq \frac{4 \varepsilon}{t}
\end{multline*}
Thus, on the set $\fS_{N,\eta}$, we can bound \eqref{eqn:doob_approximation_final_difference} from above by
\begin{multline*}
\int_0^t \left|\frac{\dd}{\dd s} h(s) (X(s)) - \sum_{i=1}^{n-1} \bONE_{\{\frac{it}{n} \leq s \leq \frac{(i+1)t}{n}\}} \frac{g_{n,i+1}\left(X\left(\frac{ti}{n}\right)\right) - g_{n,i}\left(X\left(\frac{ti}{n}\right)\right)}{t/n} \right| \dd s \\
\leq 5 \varepsilon.
\end{multline*}
We conclude that on $\fS_\eta$ we have $\sup_{X \in \fS_\eta} \left| G_n(X) - H(X) \right| \leq 10 \varepsilon$.
\end{proof}

Denote by $\bQ[G_n]$ the path-space measure obtained via the change of measure $\frac{\dd \bQ[G_n]}{\dd \PR}(X) = e^{G_n(X)}$ as in \eqref{eqn:QGn}.

\begin{proposition} \label{proposition:Doob_transform_approximation_of_entropy}
Let $\bQ[G_n]$ the path-space measure obtained via the change of measure $\frac{\dd \PR[G_n]}{\dd \PR}(X) = e^{G_n(X)}$ as in \eqref{eqn:QGn}, and let $\bQ^h$, be the measure obtained by the change of measure $\frac{\dd \bQ ^h}{\dd \PR}(X) = e^{H(X)}$ as in Lemma \ref{lemma:lvl2_Vt_in_terms_of_restricted_entropy}.

Then we have
\begin{equation*}
S(\bQ[G_n]) \rightarrow S(\bQ^h), \qquad H(\bQ^h \, | \, \bQ[G_n]) \rightarrow 0.
\end{equation*}
\end{proposition}

\begin{proof}
We start with the first claim. Note that because the time $0$ marginals of $\bQ[G_n]$ and $\bQ^h$ both equal $\PR$, it suffices to prove that
\begin{equation*}
H(\bQ[G_n] \, | \, \PR) \rightarrow H(\bQ^h \, | \, \PR).
\end{equation*}

By Lemma \ref{lemma:approx_Doob_transform_uniform_bound}, we have $\sup_n \sup_x |G_n(x)| + |H(x)| \leq M$. The restriction of the exponential map $\exp$ to $[-M,M]$ is uniformly continuous. Thus, we can find for every $\delta >0$ a $\eta >0$ such that if $|a-b| < \eta$, then $|\exp\{a\} - \exp\{b\}| < \delta$.

Thus, for an arbitrary $\delta > 0$ and corresponding $\eta \leq \delta$, we can find $N$ sufficiently large such that for $n \geq N$ the results of Lemma \ref{lemma:approx_Doob_transform_uniform_bound} and Proposition \ref{proposition:approx_of_doob_transform_analytic} hold. Then, we have
\begin{align*}
& \left| H(\bQ[G_n] \, | \, \PR) - H(\bQ^h \, | \, \PR)\right|  \\
& \leq \int \left| \frac{\dd \bQ[G_n]}{\dd \PR} \log \frac{\dd \bQ[G_n]}{\dd \PR} - \frac{\dd \bQ^h}{\dd \PR} \log \frac{\dd \bQ^h}{\dd \PR} \right| \dd \PR \\
& = \int \left|\frac{\dd \bQ[G_n]}{\dd \PR} - \frac{\dd \bQ^h}{\dd \PR}\right| \left|\log \frac{\dd \bQ[G_n]}{\dd \PR} \right| + \left| \frac{\dd \bQ^h}{\dd \PR} \right| \left|\log  \frac{\dd \bQ[G_n]}{\dd \bQ^h}\right| \dd \PR \\
& \leq \eta e^{2M} M +(1-\eta) \delta M + 2\eta Me^M + (1-\eta) \eta e^M \\
& \leq \delta e^{2M} M + \delta M + 2\delta Me^M + \delta e^M
\end{align*}
where we have bounded the contributions in line three on the set $\fS_{N,\eta}^c$ and $\fS_{N,\eta}$ separately. As $\delta > 0$ was arbitrary, the first claim is proven. The second claim follows similarly.
\end{proof}

\begin{lemma} \label{lemma:entropy_represented_as_Lagrangian}
Fix $n$, and consider the measure $\bQ[G_n]$ defined in \eqref{eqn:QGn}. Then
\begin{equation*}
S(\bQ[G_n] \, | \, \PR) = \int_0^t \cL(\gamma_n(s),\dot{\gamma}_n(s)) \dd s,
\end{equation*}
where $\gamma_n(s)$ is the law of $X(s)$ under $\bQ[G_n]$ and where $\dot{\gamma}_n(s) = (A^g_{n,i})'(\gamma_n(s))$ for $\frac{t(i-1)}{n} < s < \frac{ti}{n}$.
\end{lemma}

\begin{proof}
As $S(\bQ[G_n] \, | \, \PR) = \int \log \frac{\dd \bQ[G_n]}{\dd \PR} \dd \bQ[G_n]$, we study $G_n(X)$. Recall that
\begin{multline*}
G_n(X) = \sum_{i = 1}^{n} g_{n,i}\left(X\left(\frac{it}{n}\right)\right) - g_{n,i}\left(X\left(\frac{(i-1)t}{n}\right)\right) \\
- \int_{\frac{(i-1)t}{n}}^{\frac{it}{n}} H g_{n,i}(X(s)) \dd s.
\end{multline*}
We add and subtract $A^{g_{n,i}}g_{n,i}$ inside the integral.
\begin{align*}
& G_n(X) \\
& = \sum_{i = 1}^{n} g_{n,i}\left(X\left(\frac{it}{n}\right)\right) - g_{n,i}\left(X\left(\frac{(i-1)t}{n}\right)\right) - \int_{\frac{(i-1)t}{n}}^{\frac{it}{n}} A^{g_{n,i}} g_{n,i}(X(s)) \dd s \\
& \qquad \qquad \qquad \qquad \qquad \qquad \qquad + \int_{\frac{(i-1)t}{n}}^{\frac{it}{n}} A^{g_{n,i}} g_{n,i}(X(s)) - H g_{n,i}(X(s)) \dd s \\
& = M[G_n](t) + \int_0^t \sum_{i = 1}^n \bONE_{\{(i-1)t/n < s \leq it/n\}} Lg_{n,i}(X(s)) \dd s, 
\end{align*}
where $s \mapsto M[G_n]_s$ is a $\bQ[G_n]$ martingale by Proposition \ref{proposition:Girsanov_transform}. Thus, integration over $\bQ[G_n]$ yields
\begin{align*}
S(\bQ[G_n] \, | \, \PR) & =  \int \log \frac{\dd \bQ[G_n]}{\dd \PR} \dd \bQ[G_n] \\
& = \int \int_0^t \sum_{i = 1}^n \bONE_{\{(i-1)t/n < s \leq it/n\}} L g_{n,i}(X(s)) \, \dd s \, \bQ[G_n](\dd X).
\end{align*}
As $Lg_{n,i}$ is non-negative, we can interchange the two integrals by Tonelli's theorem, to obtain
\begin{align*}
S(\bQ[G_n] \, | \, \PR) & = \int_0^t \int \sum_{i = 1}^n \bONE_{\{(i-1)t/n < s \leq it/n\}} L g_{n,i}(X(s)) \,  \bQ[G_n](\dd X) \, \dd s \\
& = \int_0^t \sum_{i = 1}^n \bONE_{\{(i-1)t/n < s \leq it/n\}} \ip{L g_{n,i}}{\gamma_n(s)} \, \dd s.
\end{align*}
By Lemma \ref{lemma:absolutely_continuous_DG}, $s \mapsto \gamma_n(s)$ is absolutely continuous and $\dot{\gamma}_n(s) = (A^{g_{n,i}})'(\gamma_n(s))$ for almost every $s \in [0,t]$. This yields by \eqref{eqn:representationL} that $\ip{L g_{n,i}}{\gamma_n(s)} = \cL(\gamma_n(s),\dot{\gamma}_n(s))$ for almost every $s$. We conclude that $S(\bQ[G_n] \, | \, \PR) = \int_0^t \cL(\gamma_n(s),\dot{\gamma}_n(s)) \dd s$.
\end{proof}

We proceed with establishing the second inequality between the two semigroups.

\begin{proposition} \label{proposition:VisV2}
For $t \geq 0$, $h \in C(E)$ and $\mu \in \cP(E)$, we have
\begin{equation*}
\mathbf{V}(t)[h](\mu) \geq \ip{V(t)h}{\mu}.
\end{equation*}
\end{proposition}

\begin{proof}
Fix $t > 0$. As $f \mapsto V(t)f$ and $f \mapsto \mathbf{V}(t)[f]$ are continuous, it suffices to prove the result for $h \in D$. Let $\PR$ be the measure of the Markov process on $D_E(\bR^+)$ with time zero marginal $\PR_0 = \mu$. By Lemma \ref{lemma:lvl2_Vt_in_terms_of_restricted_entropy}, we find
\begin{equation} \label{eqn:proof_second_inequality_Doob_transform_representation}
\ip{V(t)h}{\mu} = \ip{V(t)h}{\PR_0} = \ip{h}{\bQ_t^h} - S(\bQ^h),
\end{equation}
where $\bQ^h$ is the measure defined by
\begin{equation*}
\frac{\dd \bQ^h}{\dd \PR}(X) = e^{h(X(t)) - \ip{V(t)h}{\PR_0}}.
\end{equation*}
Consider the approximating measures $\bQ[G_n]$. By Proposition \ref{proposition:Doob_transform_approximation_of_entropy}, we find that
\begin{enumerate}[(a)]
\item $S(\bQ[G_n]) \rightarrow S(\bQ^h)$,
\item $H(\bQ^h \, | \, \bQ[g_n]) \rightarrow 0$.
\end{enumerate}
Theorem 3.7.8 in \cite{EK86}, the fact that convergence in total variation implies weak convergence and Pinsker's inequality, together with (b), imply that $\bQ[G_n]_t$ converges weakly to $\bQ^h_t$. Secondly, proposition \ref{lemma:entropy_represented_as_Lagrangian} gives that
\begin{equation*}
S(\bQ[G_n]) = H(\bQ[G_n] \, | \, \PR) = \int_0^t \cL(\gamma_n(s),\dot{\gamma}_n(s)) \dd s.
\end{equation*}
Combining these two statements with (a) and \eqref{eqn:proof_second_inequality_Doob_transform_representation}, we find that
\begin{equation*}
\ip{V(t)h}{\mu} = \lim_{n \rightarrow \infty} \ip{f}{\gamma_n(t)} - \int_0^t \cL(\gamma_n(s),\dot{\gamma}_n(s)) \dd s.
\end{equation*}
Because
\begin{equation*}
\sup_n \int_0^t \cL(\gamma_n(s),\dot{\gamma}_n(s)) \dd s = \sup_n H(\bQ[G_n] \, | \PR) < \infty,
\end{equation*}
and $\{\gamma_n(s)\}_{s \in [0,t]} \in C_{\cP(E)}([0,T])$, we can find a converging subsequence in $C_{\cP(E)}([0,T])$ with limit $\{\gamma(s)\}_{s \in [0,t]}$ by Proposition \ref{prop:compacttimelevelsets} that has Lagrangian cost
\begin{equation*}
\int_0^t \cL(\gamma(s),\dot{\gamma}(s)) \dd s \leq \lim_{n \rightarrow \infty} \int_0^t \cL(\gamma_n(s),\dot{\gamma}_n(s)) \dd s. 
\end{equation*}
Combined with the fact that $\gamma_n(t) \rightarrow \gamma(t)$ weakly, we conclude that
\begin{align*}
\ip{V(t)h}{\mu} & = \lim_{n \rightarrow \infty} \ip{f}{\gamma_n(t)} - \int_0^t \cL(\gamma_n(s),\dot{\gamma}_n(s)) \dd s \\
& \leq \ip{h}{\gamma(t)} - \int_0^t \cL(\gamma(s),\dot{\gamma}(s)) \dd s \leq \mathbf{V}(t)[h](\PR_0).
\end{align*}
\end{proof}

\subsection{The Lagrangian form of the rate function} \label{section:Lagrangian_form}

We conclude that Propositions \ref{proposition:VisV1} and \ref{proposition:VisV2} that $\mathbf{V}(t)[f](\mu) = \ip{V(t)f}{\mu}$ for all $t \geq 0$, $f \in C(E)$ and $\mu \in \cP(E)$. We use this identification to prove $I_{t}(\mu_1 \, | \, \mu_0) = \sup_{f \in C_0(E)} \left\{\ip{f}{\mu_1} - \ip{V(t)f}{\mu_2}\right\}$ can be re-expressed using the Lagrangian.

\begin{lemma} \label{Lem:expressionforIt}
Under Condition \ref{condition:topology_on_D}, it holds that 
\begin{equation*}
I_{t}(\mu_1 \, | \, \mu_0) = \inf_{\substack{\nu \in \cA\cC_{\mu_0} \\ \nu(t) = \mu_1}} \int_{0}^{t} \cL(\nu(s),\dot{\nu}(s))  \dd s.
\end{equation*}
\end{lemma}

The proof is a classical proof using convex duality.

\begin{proof}
For a fixed measure $\mu_0 \in \cP(E)$, consider the function $\bL_{\mu_0} : \cP(E) \rightarrow [0,\infty]$ defined by
\begin{equation*}
\mathbb{L}_{\mu_0}(\mu_1) := \inf_{\substack{\nu \in \cA\cC_{\mu_0} \\ \nu(t) = \mu_1}} \int_{0}^{t} \cL(\nu(s),\dot{\nu}(s)) \dd s
\end{equation*} 
Our goal is to prove that $I_{t}(\mu_1 \, | \, \mu_0) = \bL_{\mu_0}(\mu_1)$ by showing that both are the Fenchel-Legendre transform of $\ip{V(t)g}{\mu_1}$. First, we will prove that $\bL_{\mu_0}$ is convex and has compact level sets. This last result implies the lower semi-continuity.

\smallskip

\textit{Step 1}. The convexity of $\bL_{\mu_0}$ follows directly from the convexity of $\cL$ and the fact that $\cA\cC$ is convex. So we are left to prove compactness of the level sets. Pick a sequence $\mu^n$ in the set $\{\mu \, | \, \mathbb{L}_{\mu_0}(\mu) \leq c\}$. We know by definition of $\mathbb{L}_{\mu_0}$ and Proposition \ref{prop:compacttimelevelsets} that there are $\nu^n \in \cK_{c,\{\mu_0\}}^t$ such that $\nu^n(0) = \mu_0$, $\nu^n(t) = \mu^n$ and 
\begin{equation*}
\int_0^t \cL(\nu^n(s),\dot{\nu}^n(s))  \dd s \leq c.
\end{equation*}
Again by Proposition \ref{prop:compacttimelevelsets}, we obtain that the sequence $\nu^n$ has a converging subsequence $\nu^{n_k}$ with limit $\nu^*$ such that
\begin{equation*}
\int_0^t \cL(\nu^*(s),\dot{\nu}^*(s)) \dd s \leq c.
\end{equation*}
Denote with $\mu^* : = \nu^*(t)$, then we know that $\nu^{n_k}(t) \rightarrow \mu^*$ and $\mathbb{L}_{\mu_0}(\mu^*) \leq c$, which implies that $\mathbb{L}_{\mu_0}(\cdot)$ has compact level sets and is lower semi-continuous.

\smallskip

\textit{Step 2}. Now that we know that $\bL_{\mu_0}$ is convex and lower semi-continuous, we are able to prove that $\bL_{\mu_0}(\cdot) = I_{t}(\cdot \, | \, \mu_0)$. 

$\bL_{\mu_0}(\cdot)$ is lower semi-continuous on $\cP(E)$ with respect to the weak topology, so extending its domain of definition to $\cM(E)$ by setting it equal to $\infty$ outside $\cP(E)$ does not change the fact that it is lower semi-continuous.

Because the dual of $(\cM(E),\text{weak})$ is $C(E)$ by the Riesz respresentation theorem and \cite[Theorem V.1.3]{Co07}, we obtain by Lemma 4.5.8 in Dembo and Zeitouni \cite{DZ98} that the Legendre transform $\sup_{g \in C(E)} \left\{\ip{g}{\mu_1} - \mathbf{V}(t) [g](\mu_0) \right\}$ of
\begin{align*}
\sup_{\mu_1} \left\{\ip{g}{\mu_1} - \bL_{\mu_0}(\mu_1) \right\} & = \sup_{\nu \in \cA\cC_{\mu_0}} \left\{\ip{g}{\nu(t)} - \int_0^{t} \cL(\nu(s),\dot{\nu}(s)) \dd s \right\} \\
& = \mathbf{V}(t) [g](\mu_0) 
\end{align*}
equals $\bL_{\mu_0}(\mu_1)$. Therefore, by Propositions \ref{proposition:VisV1} and \ref{proposition:VisV2}, we see
\begin{equation} \label{eqn:FL1}
\bL_{\mu_0}(\mu_1) = \sup_{g \in C_0(E)} \left\{\ip{g}{\mu_1} - \ip{V(t)g}{\mu_0} \right\}.
\end{equation}
On the other hand, by Theorem \ref{The:LDP1},
\begin{equation} \label{eqn:FL2}
I_{t}(\mu_1 \, | \,\mu_0) = \sup_{g \in C_0(E)} \left\{\ip{g}{\mu_1} - \ip{V(t)g}{\mu_0} \right\}.
\end{equation}
The combination of Equations \eqref{eqn:FL1} and \eqref{eqn:FL2}, i.e. both are the Legendre-Fenchel transform of $\ip{V(t)g}{\mu_0}$, yields that
\begin{equation*}
I_{t}(\mu_1 \, | \,\mu_0) = \bL_{\mu_0}(\mu_1) = \inf_{\substack{\nu \in \cA\cC_{\mu_0} \\ \nu(t) = \mu_1}} \int_{0}^{t} \cL(\nu(s),\dot{\nu}(s)) \dd s.
\end{equation*}
\end{proof}

We proceed with the final lemma before the proof of Theorem \ref{The:LDP2}.

\begin{lemma} \label{lemma:Jisgood}
The function $J : C_{\cP(E)}(\bR^+) \rightarrow [0,\infty]$, given by
\begin{equation*}
J(\mu) = \begin{cases}
H(\mu(0) \, | \, \PR_0) + \int_0^\infty \cL(\mu(s),\dot{\mu}(s)) \dd s  & \text{if }\mu \in \cA\cC,  \\
\infty & \text{otherwise},
\end{cases}
\end{equation*}
has compact level sets in $C_{\cP(E)}(\bR^+)$.
\end{lemma}

\begin{proof}
Clearly, $\{J \leq M\} \subseteq \bigcap_T \cK_{M}^T$. So, pick a sequence $\mu^n \in \{J \leq M\}$. For $n = 1$, we can construct a converging subsequence $\mu^{n_{k}}$ in $\cK_{M}^1$ seen as a subset of $C_{\cP(E)}([0,1])$. From this subsequence, we can extract yet another subsequence that has the same property on $[0,2]$. By a diagonal argument, this yields a converging subsequence in $C_{\cP(E)}(\bR^+)$. By the lower semi-continuity of $H(\cdot \, | \PR_0)$ and $\cL$ this yields that the limit is in $\{J \leq M\}$.
\end{proof}

\begin{proof}[Proof of Theorem \ref{The:LDP2}]
By using the contraction principle for the identity map $C_{\cP(E)}(\bR^+) \rightarrow \prod_{\bR^+} \cP(E)$, we find that the rate function in Theorem \ref{The:LDP1} coincides with the rate function which would have been found via the Dawson-Gärtner theorem \cite[Theorem 4.6.1]{DZ98} for the large deviation problem on $\prod_{\bR^+} \cP(E)$.

\smallskip

On the other hand, the rate function of Theorem \ref{The:LDP1} for the finite dimensional distributions at times $0 = t_0 < t_1 < \dots < t_k$ is given by
\begin{equation*}
I[t_0,\dots, t_k](\mu(0),\dots, \mu(t_k)) := H(\mu(0) \, | \PR) + \sum_{i=1}^k I_{t_i- t_{i-1}}(\mu(t_i) \, | \, \mu(t_{i-1})).
\end{equation*}
By Lemma \ref{Lem:expressionforIt}, this can be rewritten as
\begin{multline} \label{eqn:finite_dim_rate_function_Lagrangian_form}
H(\mu(0) \, | \PR) + \sum_{i=1}^k \inf_{\substack{\nu \in \cA\cC \\ \nu(t_{k-1}) = \mu_{t_{k-1}} \\ \nu(t_k) = \mu_{t_k}}} \int_{t_{k-1}}^{t_k} \cL(\nu(s),\dot{\nu}(s))  \dd s \\
= H(\mu(0) \, | \PR) +  \inf \left\{\int_{0}^{\infty} \cL(\nu(s),\dot{\nu}(s))  \dd s \, \middle| \, \nu \in \cA\cC \,  \forall  i: \nu(t_i) = \mu(t_i) \right\}.
\end{multline}
In this context, we can apply Lemma 4.6.5 from \cite{DZ98} to find that if we have a good rate function $J$ on $\prod_{\bR^+} \cP(E)$ that satisfies
\begin{equation} \label{equation:ratefunctionviainf}
I\left[0,t_1,\dots,t_k\right](\mu(0), \mu(t_1),\dots,\mu(t_k))  = \inf\left\{J(\nu) \, \middle| \, \forall i: \nu(t_i) = \mu(t_i) \right\},
\end{equation}
then it holds that $I = J$. The candidate
\begin{equation*}
J(\mu) = \begin{cases}
H(\mu(0) \, | \PR_0) + \left\{\int_0^\infty \cL(\mu(s),\dot{\mu}(s)) \dd s \right\} & \text{if }\mu \in \cA\cC,  \\
\infty & \text{otherwise},
\end{cases}
\end{equation*}
satisfies Equation \eqref{equation:ratefunctionviainf} in view of \eqref{eqn:finite_dim_rate_function_Lagrangian_form}. By Lemma \ref{lemma:Jisgood}, we know that $J$ is a good rate function on $C_{\cP(E)}(\bR^+)$ and therefore also on $\prod_{\bR^+} \cP(E)$.
\end{proof}

\subsection{Preparations for the proof of Proposition \ref{prop:compacttimelevelsets}} \label{section:preparations_proof_compact_level_sets}

We say that a topological space is Souslin if it is the continuous image of a complete separable metric space. For the proof of Proposition \ref{prop:compacttimelevelsets}, we will need the generalisation of one of the implications of the Prohorov theorem.

\begin{theorem}[Prohorov] \label{theorem:Prohorovtheorem}
Let $\cK$ be a subset of the Borel measures on a completely regular Souslin space $\cS$ that is uniformly bounded with respect to the total variation norm. If $\cK$ is a tight family of measures, then $\cK$ has a compact and sequentially compact closure with respect to the weak topology on $\cP(\cS)$. 
\end{theorem}

The Prohorov theorem is given in \cite[Theorem 8.6.7]{Bo07} and its specialisation to completely regular Souslin spaces follows from \cite[Corollary 6.7.8 and Theorem 7.4.3]{Bo07}. Note that the other implication of the ordinary Prohorov theorem does not necessarily hold in this generality \cite[Proposition 8.10.19]{Bo07}.

We will use the Prohorov theorem for measures on the product space $(\cP(E)\times U \times [0,T])$, where the first two spaces are equipped with the weak* topology, and the last space with its standard topology. 

\begin{lemma} \label{lemma:Prohorov_can_be_used}
The space $(\cP(E)\times U \times [0,T])$ is completely regular and Souslin.
\end{lemma}

\begin{proof}
Because taking products and subspaces preserves complete regularity, the first claim follows by establishing complete regularity for $(D',wk^*)$. This follows from Lemma \cite[15.2.(3)]{Ko69}.

The Souslin property follows because $(U,wk^*)$ is Souslin by Condition \ref{condition:topology_on_D} (a) and Lemma \ref{lemma:UisSouslin}, and because the product of Souslin spaces is Souslin, \cite[Lemma 6.6.5]{Bo07}.
\end{proof}

Suppose that we have a weakly converging net of measures on $(\cP(E)\times U \times [0,T])$. By definition, integrals of continuous and bounded functions with respect to this net of measures converges in $\bR$. The next lemmas are aimed to extend this property to continuous functions, that are unbounded, but linear on $U$.

\begin{definition}
For the neighbourhood $\cN$, we define the Minkowski functional $\vn{\cdot}_U$ on $U$ by $\vn{u}_\cN := \inf \left\{c \geq 0 \, \middle| \, u \in c \cN^\circ \right\}$.
\end{definition}

We have the following elementary results.

\begin{lemma} \label{lemma:seminorm_is_achieved}
$\vn{\cdot}_{\cN}$ is a norm on $U$,
$\{u \, | \, \vn{u}_\cN \leq 1\} = \cN^\circ$. Furthermore, for $u \in U$, we have
\begin{equation*}
\sup_{f \in c \cN} \frac{\ip{f}{u}}{\vn{u}_{\cN}} = c.
\end{equation*}
\end{lemma}

We use this lemma to find functions $\phi$ of the type given in the following lemma, which is an analogue of the de la Vall\'{e}e-Poussin lemma \cite[Theorem 4.5.9]{Bo07} and can be proven similarly. 

\begin{lemma} \label{lemma:ValleePoussin}
Let $\{\pi^\alpha\}$ be a collection of measures on some measurable space that is bounded in total variation norm. Let $f$ be a measurable function and suppose that there exists a non-negative non-decreasing function $\phi : \bR^+ \rightarrow \bR^+$ which satisfies $\lim_{r \rightarrow \infty} r^{-1}\phi(r) = \infty$ and for which it holds that $\sup_\alpha \int \phi(|f|) \dd \pi^\alpha \leq M < \infty$. Then it holds that
\begin{equation*}
\sup_\alpha \int |f| \dd \pi^\alpha < \infty.
\end{equation*}
Also, we obtain that 
\begin{equation} \label{equation:ValleePoussin}
\lim_{C \rightarrow \infty} \sup_{\alpha} \int \left|f - \Upsilon_C(f)\right|\dd \pi^\alpha  = 0,
\end{equation}
where $\Upsilon_C(f) = (f \vee -C) \wedge C$.
\end{lemma}

\begin{lemma} \label{lemma:norm_bounded_by_Lagrangian}
Under Condition \ref{condition:topology_on_D} (e) that states that for every $c \geq 0$: $\Gamma(c) := \sup_{f \in c \cN} \vn{Hf} < \infty$, there exists an increasing function $\phi : \bR^+ \rightarrow \bR^+$, such that $\lim_{r \rightarrow \infty} r^{-1}\phi(r) = \infty$ and such that $\phi(|\ip{f}{u}|) \leq \phi(\vn{u}_\cN) \leq \cL(\mu,u)$ for every $f \in \cN$, $u \in U$ and $\mu \in \cP(E)$.
\end{lemma}
The proof of this lemma is inspired by the proof of Lemma 10.21 in Feng and Kurtz \cite{FK06}.

\begin{proof}
For $u \neq 0$ in $U$, Lemma \ref{lemma:seminorm_is_achieved} yields
\begin{equation*}
\frac{\cL(\mu,u)}{\vn{u}_\cN}  \geq \sup_{f \in c \cN} \left\{\frac{\ip{f}{u}}{\vn{u}_\cN} - \frac{\ip{Hf}{\mu}}{\vn{u}_\cN} \right\}  \geq c - \frac{\Gamma(c)}{\vn{u}_\cN}
\end{equation*}
for every $c > 0$. This directly yields for every $c > 0$
\begin{equation*}
\lim_{r \rightarrow \infty} \inf_{\mu \in \cP(E)} \inf_{u \, : \, \vn{u}_\cN \geq r} \frac{\cL(\mu,u)}{\vn{u}_\cN}  \geq \lim_{r \rightarrow \infty} \inf_{\mu \in \cP(E)} \inf_{u \, : \, \vn{u}_\cN \geq r} c - \frac{\Gamma(c)}{\vn{u}_\cN} = c,
\end{equation*}
which implies
\begin{equation*}
\lim_{r \rightarrow \infty} \inf_{\mu \in \cP(E)} \inf_{u \, : \, \vn{u}_\cN \geq r} \frac{\cL(\mu,u)}{\vn{u}_\cN} = \infty.
\end{equation*}
Consequently, the function
\begin{equation*}
\phi(r) = r \inf_{\mu \in \cP(E)} \inf_{u \, : \, \vn{u}_\cN \geq r} \frac{\cL(\mu,u)}{\vn{u}_\cN}, 
\end{equation*}
satisfies the claims in the lemma.
\end{proof}

\subsection{Proof of Proposition \ref{prop:compacttimelevelsets}} \label{section:proof_compact_level_sets}

We now have the tools for the proof of Proposition \ref{prop:compacttimelevelsets}. Essentially, the proof follows the approach as in Feng and Kurtz \cite[Proposition 8.13]{FK06}. We give it for clarity and completeness as there are some notable differences. First of all, we work with absolutely continuous paths, instead of paths that satisfy a relaxed control equation. Second, the possible `speeds' that we allow are elements of the completely regular Souslin subset $U$ of a locally convex space instead of a metric space.

\begin{proof}[Proof of Proposition \ref{prop:compacttimelevelsets}] 
Pick a sequence $\mu^n \in \cK_{M}^T$. As $\cP(E)$ is compact, we assume that $\mu^n(0) \rightarrow \mu_0$. Define the occupation measures $\pi^n$ on $\cP(E) \times U \times [0,T] \subseteq \cP(E) \times U \times [0,T]$ by
\begin{equation*}
\pi^n(B \times [0,t]) = \int_0^t \bONE_B(\mu^n(s),\dot{\mu}^n(s)) \dd s.
\end{equation*}
Proposition \ref{Prop:FKcompactness}  tells us that $\pi^n$ is tight in $\cP\left(\cP(E) \times U \times [0,T]\right)$ by considering the following calculation:
\begin{multline*}
C \pi^n \left\{(\mu,u,t) \in \cP(E) \times U \times [0,T] \, \middle| \, \cL(\mu,u) \leq C  \right\}^c \\
\leq \int_0^T \cL(\mu,u) \pi^n(\dd \mu \times \dd u \times \dd s) \leq M.
\end{multline*}
In other words
\begin{equation} \label{eqn:compactset}
\pi^n \left\{(\mu,u,t) \in \cP(E)\times U \times [0,T] \, \middle| \, \cL(\mu,u) \leq C  \right\}^c \leq \frac{M}{C},
\end{equation}
and because $C$ is arbitrary, we can choose it big enough such that this probability is smaller then any $\varepsilon > 0$ uniformly in $n$. This implies by Theorem \ref{theorem:Prohorovtheorem} that $\pi^n$ contains a weakly converging subsequence.  Therefore, we assume without loss of generality that, there exists $\pi \in \cP(\hat{K} \times U \times [0,T])$ such that $\pi^n \rightarrow \pi$ weakly. 

\smallskip

We now show that $\pi$ gives us a new path $s \mapsto \mu(s)$ in $\cK_{M}^T$.
Recall that for $c \geq 0$ $\Upsilon_c(g) = (g \wedge c) \vee -c$. So for a fixed $f \in D$ (we can take $f \in \cN$ without loss of generality as $\cN$ is a barrel), $u \mapsto \Upsilon_c(\ip{f}{u})$ is a bounded and continuous function. For an arbitrary $t \leq T$, the set $\pi(\cP(E)\times U \times \{t\})$ is a set of measure $0$, so the function  $(u,s) \mapsto \bONE_{\{s \leq t\}} \Upsilon_c(\ip{f}{u})$ is a bounded Borel measurable functions that is continuous $\pi$ almost everywhere.

Hence, by the weak convergence of $\pi^n$ to $\pi$ and Corollary 8.4.2 in Bogachev \cite{Bo07}, we obtain for every $c \geq 0$ that
\begin{equation} \label{eqn:compactness_of_curves_eq1}
\int_{\{s \leq t\}} \Upsilon_c(\ip{f}{u}) \; \pi^n(\dd \mu \times \dd u \times \dd s) \rightarrow \int_{\{s \leq t\}} \Upsilon_c(\ip{f}{u}) \; \pi(\dd \mu \times \dd u \times \dd s).
\end{equation}

By the Portmanteau theorem and the lower semi-continuity of $\cL$, we obtain that
\begin{equation*}
\int \cL(\mu,u) \; \pi(\dd \mu \times \dd u \times \dd s) \leq \liminf_n \int \cL(\mu,u) \; \pi^n(\dd \mu \times \dd u \times \dd s) \leq M.
\end{equation*}
As $\phi(|\ip{f}{u}|) \leq \cL(\mu,u)$ by Lemma \ref{lemma:norm_bounded_by_Lagrangian}, and the fact that $\phi$ satisfies the conditions of Lemma \ref{lemma:ValleePoussin}, we use the result in \eqref{equation:ValleePoussin} to obtain that 
\begin{equation} \label{eqn:compactness_of_curves_eq2}
\sup_n \left|\int_{\{s \leq t\}} \ip{f}{u} \; \pi^n(\dd \mu \times \dd u \times \dd s) - \int_{\{s \leq t\}} \Upsilon_c(\ip{f}{u}) \; \pi^n(\dd \mu \times \dd u \times \dd s) \right|  \rightarrow 0,
\end{equation}
as $c \rightarrow \infty$. This also follows for the limiting measure $\pi$:
\begin{equation} \label{eqn:compactness_of_curves_eq3}
\left|\int_{\{s \leq t\}} \ip{f}{u} \; \pi(\dd \mu \times \dd u \times \dd s) - \int_{\{s \leq t\}}  \Upsilon_c(\ip{f}{u}) \; \pi(\dd \mu \times \dd u \times \dd s) \right|  \rightarrow 0.
\end{equation}

Using the triangle inequality, Equations \eqref{eqn:compactness_of_curves_eq1}, \eqref{eqn:compactness_of_curves_eq2} and \eqref{eqn:compactness_of_curves_eq3}, sending first $c$ and then $n$ to infinity, we get
\begin{equation} \label{eqn:linear_functions_converge}
\left| \int_{\{s \leq t\}} \ip{f}{u} \; \pi^n(\dd \mu \times \dd u \times \dd s) - \int_{\{s \leq t\}} \ip{f}{u} \; \pi(\dd \mu \times \dd u \times \dd s) \right| \rightarrow 0
\end{equation}

Fix some $0 \leq t \leq T$ and pick a sequence $0 \leq t_n \leq T$ that converges to $t$. Because $\mu^n(t_n)$ is a sequence in the compact set $\cP(E)$ it has a converging subsequence with limit $\nu$. By Lemmas \ref{lemma:ValleePoussin}, \ref{lemma:norm_bounded_by_Lagrangian}, and the Dominated convergence theorem, we have
\begin{equation*}
\lim_{n\rightarrow \infty} \int \bONE\{s \text{ between } t_n \text{ and } t\} |\ip{f}{u}|  \pi^n(\dd \mu \times \dd u \times \dd s) \rightarrow 0,
\end{equation*}
which implies, using Equation \eqref{eqn:linear_functions_converge}, that
\begin{align*}
\ip{f}{\nu} - \ip{f}{\mu_0} & = \lim_n \ip{f}{\mu^n(t_n)} - \ip{f}{\mu^n(0)} \\
& = \lim_n \int \bONE\{s \leq t\} \ip{f}{u}  \pi^n(\dd \mu \times \dd u \times \dd s) \\
& \qquad -  \int \bONE\{s \text{ between } t_n \text{ and } t\} \ip{f}{u}  \pi^n(\dd \mu \times \dd u \times \dd s) \\
& = \int \bONE\{s \leq t\} \ip{f}{u}  \pi(\dd \mu \times \dd u \times \dd s).
\end{align*}
As $D$ is dense in $C(E)$, this uniquely determines $\nu$, and for every sequence $s_n \rightarrow t$, one gets $\mu^n(s_n) \rightarrow \nu$ weakly. Therefore, we will denote $\mu(t) := \nu$. This way, we can construct $\mu(t)$ for a countable dense subset $J$ of $[0,T]$ and $\mu(t)$ is continuous on $J$. As a consequence, $\mu(t)$ extends continuously to $[0,t]$ and satisfies
\begin{equation*}
\ip{f}{\mu(t)} - \ip{f}{\mu_0} = \int \bONE_{\{s \leq t\}} \ip{f}{u} \pi(\dd \mu \times \dd u \times \dd s)
\end{equation*}
for every $f \in D$. This implies that for any sequence $s_n \rightarrow t$, we have $\mu(s_n) \rightarrow \mu(t)$, which yields that $\left\{\mu^n(t)\right\}_{0\leq t \leq T}$ converges to $\left\{\mu(t)\right\}_{0\leq t \leq T}$ in $C_{\cP(E)}([0,T])$.

\smallskip

We proceed with extracting the speed of the trajectory $s \mapsto \mu(s)$ from the measure $\pi$. Let $\hat{\pi}$ be the measure $\pi$ restricted to $U \times [0,T]$. By Corollary 10.4.6 in Bogachev \cite{Bo07}, we can write $\hat{\pi}(\dd u \times \dd s)$ as $\lambda_s(\dd u) \dd s$. 

For Lebesgue almost every $s$, we know that $\int |\ip{f}{u}| \lambda_s(\dd u ) < \infty$, so we can define the Gelfand integral $\bar{u}(s) = \int u \lambda_s(\dd u)$, see Theorem \ref{theorem:Gelfandrepresentation}. We show that $\bar{u}(s) = \dot{\mu}(s)$. First, by the measurability of $s \mapsto \lambda_s$, also $s \mapsto \bar{u}$ is measurable. Second, by Jensen's inequality in the first line, and the lower semi-continuity of $\cL$ in the third,
\begin{align*}
\int_0^T |\ip{f}{\bar{u}(s)}| \dd s & \leq \int |\ip{f}{u}| \pi(\dd \mu \times \dd u \times \dd s) \\
& \leq T(\vn{Hf} \vee \vn{H(-f)}) + \int \cL(\mu,u) \pi(\dd \mu \times \dd u \times \dd s) \\
& \leq  T(\vn{Hf} \vee \vn{H(-f)}) + \liminf_n \int \cL(\mu,u) \pi^n(\dd \mu \times \dd u \times \dd s) \\
& \leq  T(\vn{Hf} \vee \vn{H(-f)}) + M.
\end{align*}
Last,
\begin{multline*}
\ip{f}{\mu(t)} - \ip{f}{\mu(0)}  = \int \bONE_{\{s \leq t\}} \ip{f}{u} \pi(\dd \mu \times \dd u \times \dd s) \\
= \int_0^t \int \ip{f}{u}\lambda_s(\dd u) \dd s  = \int_0^t \ip{f}{\bar{u}(s)} \dd s.
\end{multline*}
This means that $\mu \in \cA\cC^T$ and $\dot{\mu} = \bar{u}$. 

We still need to show that $\mu \in \cK_{M}^T$ by showing that its Lagrangian cost is bounded by $M$. By the construction of the path $s \mapsto \mu(s)$, it is clear that we have $\pi(\dd \mu \times \dd u \times \dd s) = \bONE_{\{s \leq T\}} \delta_{\{\mu(s)\}}(\dd \mu)\lambda_s(\dd u) \dd s$.
This shows, using the convexity of $\cL$ in the second line, and lower semi-continuity of $\cL$ in the third line, that
\begin{align*}
\int_0^T \cL(\mu(s),\dot{\mu}(s)) \dd s & = \int \cL(\mu,u) \bONE\{s \leq T\}  \delta_{\mu(s)}(\dd \mu) \delta_{\bar{u}(s)}(\dd u) \dd s \\
& \leq \int  \cL(\mu,u) \bONE\{s \leq T\} \delta_{\mu(s)}(\dd \mu) \lambda_s(\dd u) \dd s \\
& \leq \liminf_n \int_0^T \cL(\mu^n(s),\dot{\mu}^n(s)) \dd s \leq M.
\end{align*}
So indeed $\cK_{M}^T$ is compact in $C_{\cP(E)}(\bR^+)$.
\end{proof}

\subsection{Proof of Proposition \ref{proposition:finite_lagrangian_cost_implies_strong_abs_cont}} \label{section:proof_abs_cont}

\begin{proof}[Proof of Proposition \ref{proposition:finite_lagrangian_cost_implies_strong_abs_cont}]

We start with the proof of (a). Let $\gamma \in \cA\cC$ be strongly absolutely continuous. Let $H : [0,\infty) \rightarrow \bR$ be the absolutely continuous function such that $\sup_{f \in \cN} \left|\ip{f}{\gamma(t)} - \ip{f}{\gamma(s)} \right| \leq \left| H(t) - H(s) \right|$. We conclude that for all $f \in \cN$, and thus for all $f \in D$ that $t \mapsto \ip{f}{\gamma(t)}$ is absolutely continuous.
Following the proof of Lemma 4.2 in \cite{DG87}, with $\cN$ instead of the collection of neighbourhoods $U_{K_n}$, we conclude that there exists a weakly measurable trajectory $s \mapsto u(s)$ in $D'$ such that for all $f$ in a countable dense subset of $D$ and for all $t$ in a subset of full measure it holds that $\frac{\dd}{\dd t} \ip{f}{\gamma(t)} = \ip{f}{u(t)}$.

By absolute continuity of $t \mapsto \ip{f}{\gamma(t)}$, we also know that $\int_0^t |\ip{f}{u(s)}| \dd s < \infty$ and additionally, 
\begin{equation*}
\ip{f}{\gamma(t)} - \ip{f}{\gamma(0)} = \int_0^t \ip{f}{u(s)} \dd s \qquad \forall \, f \in D, \forall \, t \geq 0.
\end{equation*}
We conclude that $\gamma \in \cA\cC$.

\smallskip

We proceed with the proof of (b). Let $\gamma \in \cA\cC$ be absolutely continuous and such that $\int_0^\infty \cL(\gamma(s),\dot{\gamma}(s)) \dd s < \infty$. Let $\phi$ be the function introduced in Lemma \ref{lemma:norm_bounded_by_Lagrangian}. Denote by $\psi := \phi^{-1}$. An elementary computation shows that $\lim_{r \rightarrow \infty} r^{-1}\psi(r)= 0$. By Lemma \ref{lemma:norm_bounded_by_Lagrangian}, we obtain that $\sup_{f \in \cN} \left|\ip{f}{u} \right| \leq \psi(\cL(\mu,u))$ for all $\mu \in \cP(E)$ and $\nu \in D'$.

\smallskip

By Condition $\ref{condition:topology_on_D}$ the space $(D,\tau_D)$ is separable. Thus, we can find a $\tau_D$ dense sequence of functions $\{f_n\}_{n \geq 1}$ in $\cN$ such that
\begin{equation*}
\sup_{f \in \cN} |\ip{f}{\gamma(s)} | = \sup_n  |\ip{f_n}{\gamma(s)}|.
\end{equation*}
We conclude that the function $h : [0,\infty) \rightarrow [0,\infty]$ defined by $h(s) := \sup_{f \in \cN} |\ip{f}{\gamma(s)}|$ is measurable. In fact $h$ is locally integrable because
\begin{align*}
\int_0^T h(s) \dd s & = \int_0^T  \sup_{f \in \cN} |\ip{f}{\gamma(s)}| \dd s \leq \int_0^T \psi(\cL(\gamma(s),\dot{\gamma}(s))) \dd s \\
& = \int_0^T (1 \vee \cL(\gamma(s),\dot{\gamma}(s)))\frac{\psi(\cL(\gamma(s),\dot{\gamma}(s)))}{1 \vee \cL(\gamma(s),\dot{\gamma}(s))} \dd s \\
& \leq \int_0^T (1 \vee \cL(\gamma(s),\dot{\gamma}(s))) M \dd s < \infty,
\end{align*}
where $M := \sup_{r \geq 1} r^{-1}\psi(r)$. Put $H(t) = \int_0^t h(s) \dd s$, which is a non-decreasing absolutely continuous function by construction. We find by the fact that $\gamma \in \cA \cC$ that for $s \leq t$ and $f \in \cN$
\begin{equation*}
\left|\ip{f}{\gamma(t)} - \ip{f}{\gamma(s)} \right| \leq \int_s^t \left| \ip{f}{\dot{\gamma}(r)} \right| \dd r \leq \int_s^t  h(r) \dd r = H(t) - H(s).
\end{equation*}
\end{proof}

\section{Examples} \label{section:examples}

We give a number of examples on which Theorem \ref{The:LDP2} can be applied. We start by considering degenerate diffusion processes on $\bR^d$. We proceed with Lévy processes on $\bR^d$. Third, we consider  Markov jump process with bounded jump rates on a locally compact separable metric space. Finally, we consider interacting particle systems\cite{Li85}. In this final case, we also prove a representation theorem for $D'$.

\subsection{Diffusion processes on \texorpdfstring{$\bR^d$}{Rd}} \label{section:diffusion_processes}

We now show that our result partly extends the large deviation result of Dawson and Gärtner theorem \cite{DG87} for the empirical density of $n$ non-interacting particles. 

We start by introducing a topology on $C_c^\infty(\bR^d)$ that is well known in the theory of distributions. Let $K_1 \subseteq K_2 \subseteq \dots$ be a sequence of compact sets in $\bR^d$ such each $K_n$ is contained in the interior of $K_{n+1}$ and such that $\bigcup_n K_n = \bR^d$.

\smallskip

Let $p = (p_1,\dots, p_d)$ be a multi-index and define $|p| = \sum p_i$. Denote by
\begin{equation*}
\left(\frac{\partial}{\partial x}\right)^p f(x) = \sum_{i = 1}^d \left(\frac{\partial}{\partial x^i}\right)^{p_i} f(x),
\end{equation*}
where $x = (x^1,\dots,x^d)$ are the standard Euclidean coordinates.

Consider the spaces $C^\infty_0(K_n)$ of smooth functions on $\bR^d$ that are supported in $K_n$ and equip it with the Fr\'{e}chet topology $\tau_{n}$ generated by all semi-norms of the type
\begin{equation*}
\vn{f}_{K_n,m} = \sum_{p: \, |p| \leq m} \sup_{x \in K_n} \left|\left(\frac{\partial}{\partial x}\right)^p f(x) \right|.
\end{equation*}

Finally, we equip $C_c^\infty(\bR^d)$ with the limit Fr\'{e}chet (LF) topology $\tau$, see for example Chapter 13 in \cite{Tr67}.

\begin{definition}
The LF topology $\tau$ on $C_c^\infty(\bR^d)$ is generated by the collection of convex sets $U$ containing $0$ such that $U \cap C_0^\infty(K_n)$ is a open neighbourhood of $0$ in $\tau_n$ for all $n$.
\end{definition}

The space $(C_c^\infty(\bR^d),\tau)$ is well known as the space of test functions. Its continuous dual space is the space of distributions. We proceed with the relation of absolute continuity in the sense of Definition 4.1 in \cite{DG87}, which has also been used in \cite{DjKa95,Le95}, with strong and weak absolute continuity in the sense of Definitions \ref{definition:absolutely_continuous} and \ref{definition:strongly_absolutely_continuous}. 

\begin{lemma} \label{lemma:absolutely_continuous_DG}
Let $\gamma \in C_\cP(E)[0,T]$. Then (a) implies (b) implies (c).
\begin{enumerate}[(a)]
\item $\gamma$ is strongly absolutely continuous in the sense of Definition \ref{definition:strongly_absolutely_continuous}.
\item $\gamma$ is absolutely continuous in the sense of Definition 4.1 in \cite{DG87}. For all compact sets $K$ there exists a $\tau$-neighbourhood $U_K \subseteq C_0^\infty(K)$ of $0$ and an absolutely continuous function $H_K : [0,T] \rightarrow \bR$ such that
\begin{equation} \label{eqn:abs_cont_DG}
|\ip{f}{\gamma(t)} - \ip{f}{\gamma(s)}| \leq |H_K(t) - H_K(s)|
\end{equation}
for all $s,t \in [0,T]$ and $f \in U_K$.
\item $\gamma$ is absolutely continuous in the sense of Definition \ref{definition:absolutely_continuous}
\end{enumerate}
\end{lemma}

Thus, if the trajectory has finite Lagrangian cost, all three notions are equivalent by Proposition \ref{proposition:finite_lagrangian_cost_implies_strong_abs_cont}

\begin{proof}[Proof of Lemma \ref{lemma:absolutely_continuous_DG}]
Let $\gamma$ satisfy (a). Recall that the inductive limit topology $\tau$ induces the Fr\'{e}chet topology $\tau_{n}$ on $C_0^\infty(K_n)$ for every $n$. Thus (b) is satisfied by taking $U_K = \cN \cap C_0(K)$ and $H_K = H$ for any compact set $K \subseteq \bR^d$. Recall for this argument that the inductive limit topology $\tau$ induces the Fr\'{e}chet topology $\tau_{n}$

Let $\gamma$ satisfy (b). Fix $f \in C_c^\infty(\bR^d)$. Suppose without loss of generality that $f \in C_0^\infty(K)$. It is not immediately clear that $f \in U_K$. However, in a locally convex space, one can always find a barrel $B \subseteq U_K$, cf. Proposition 7.2 in \cite{Tr67}. Because barrels are absorbing, there is some $\lambda > 0$ such that $f \in \lambda B \subseteq \lambda U_K$. We conclude that \eqref{eqn:abs_cont_DG} holds for $f$ with $\lambda H_K$ instead of $H_K$.

This means that $t \mapsto \ip{f}{\gamma(t)}$ is absolutely continuous. We conclude that it is differentiable almost everywhere. By Lemma 4.2 in \cite{DG87}, this derivative equals $\ip{f}{\dot{\gamma}}$ almost everywhere. We conclude that (c) is satisfied.
\end{proof}

As we will use this space for diffusion and for Lévy processes, we check the process-independent conditions for the large deviation theorem directly.

\begin{lemma} \label{lemma:properties_test_functions}
Conditions \ref{condition:D_algebra_closed_smooth} and \ref{condition:topology_on_D} (a)-(c) are satisfied for $(C_c^\infty(\bR^d),\tau)$.
\end{lemma}

\begin{proof}
It is clear that $C_c^\infty(\bR^d)$ is an algebra that is closed under composition with smooth functions. Additionally, it is clear that these operations are continuous for $\tau$.

\smallskip

We proceed with proving \ref{condition:topology_on_D} (a). By Corollary 33.3  in \cite{Tr67} the space $(C_c^\infty(\bR^d),\tau)$ is barrelled. By the Remark following Proposition A.9 in \cite{Tr67} the space $(C_c^\infty(\bR^d),\tau)$ is Souslin, which in particular implies that it is separable.

\smallskip

We are left to prove \ref{condition:topology_on_D} (b). We have to prove that the embedding $\iota : C_c^\infty(\bR^d) \rightarrow C_0(\bR^d)$ is $\tau$ to $\vn{\cdot}$ continuous. By Proposition 14.7 in \cite{Tr67} it suffices to prove sequential continuity. Furthermore, by Corollary 14.1 in \cite{Tr67} a sequence in $(C_c^\infty(\bR^d),\tau)$ converges if and only if it is contained in $(C^\infty_0(K_n),\tau_n)$ and converges there. Thus the sequential continuity and thus continuity of $\iota$ follows.
\end{proof}

We proceed with the large deviations of diffusion processes on $\bR^d$. Let $S^d$ be the space of $d \times d$ nonnegative-definite matrices. We give two generation theorems for diffusion processes with generator $(A,C_c^\infty(\bR^d))$ defined by
\begin{equation} \label{eqn:diffusion_generator}
Af(x) = \frac{1}{2}\sum_{i,j = 1}^d a^{ij}(x) \frac{\partial^2 f}{\partial x^i \partial x^j}(x) + \sum_{i=1}^d b^i(x)\frac{\partial f}{\partial x^i}(x).
\end{equation}

The first theorem considers non-degenerate diffusion matrices, the second one considers degenerate diffusion matrices. 

\begin{theorem}[Theorem 8.1.7 in \cite{EK86}] \label{theorem:construction_of_diffusion_processes1}
Let $a : \bR^d \rightarrow S^d$ and $b : \bR^d \rightarrow \bR^d$ be bounded. Let $0 < \mu \leq 1$, $K > 0$ and suppose that
\begin{equation*}
|a(x) - a(y)| + |b(x) - b(y)| \leq K|x-y|^\mu,
\end{equation*}
and
\begin{equation*}
\inf_{x \in \bR^d} \inf_{|\theta| = 1} \ip{\theta}{a(x) \theta} > 0.
\end{equation*}
Then the closure of $(A,C_c^\infty(\bR^d))$ defined in \eqref{eqn:diffusion_generator} generates a Feller process.
\end{theorem}

\begin{theorem}[Theorem 8.2.5 in \cite{EK86}] \label{theorem:construction_of_diffusion_processes2}
Let $a : \bR^d \rightarrow S^d$ be such that $x \mapsto a^{ij}(x)$ is twice continuously differentiable(possible non-bounded) and such that for all $i,j,k,l$ the map $x \mapsto \frac{\partial^2 a^{ij}}{\partial x^k \partial x^l} (x)$ is bounded. Let $b : \bR^d \rightarrow \bR$ be Lipschitz continuous. Then the closure of $(A,C_c^\infty(\bR^d))$ defined in \eqref{eqn:diffusion_generator} generates a Feller process.
\end{theorem}

An elementary calculation yields
\begin{equation} \label{equation:definition_H_diffusion}
H f(x) = A f(x) + \frac{1}{2} \sum_{i,j} a^{ij}(x) \frac{\partial f}{\partial x^i}(x) \frac{\partial f}{\partial x^j}(x).
\end{equation}

We check the conditions for the large deviation result.

\begin{theorem} \label{theorem:condition_diffusion}
Let $A$ satisfy the conditions Theorem \ref{theorem:construction_of_diffusion_processes1} or Theorem \ref{theorem:construction_of_diffusion_processes2}. Then $(C_c^\infty(\bR^d),\tau)$ and $H$ satisfy Conditions and \ref{condition:D_algebra_closed_smooth} and \ref{condition:topology_on_D}. As a consequence, Theorems \ref{The:LDP1} and \ref{The:LDP2} hold for iid copies of a diffusion process with generator $A$.
\end{theorem}

Below, in Proposition \ref{proposition:representation_of_rate_function_diffusions}, we give a representation of the rate function in terms of an inverse Sobolev space norm. 

Compared to the result for the trajectories of copies of independent diffusion processes of Dawson and Gärtner, there are three main differences.
\begin{enumerate}[(a)]
\item Our results only hold in the time-homogeneous setting.
\item We are not restricted to the case where the diffusion matrices are positive definite.
\item The assumption that the compactly supported smooth functions are a core is more stringent than the condition that the martingale problem is well-posed. 
\end{enumerate}

Then there is the result by Feng and Kurtz, Section 13.3 in \cite{FK06}, where the large deviation principle is established under the condition that the drift term is twice continuously differentiable, is semi-convex, and grows sufficiently fast at infinity. The diffusion matrix is assumed to be the identity matrix. Under these conditions it is shown that the trajectories satisfy the large deviation principle in $D_{\cP^2(E)}(\bR^+)$, where $\cP^2(E)$ is the space of probability measures on $E$ with bounded second moments equipped with the Kantorovich-Wasserstein 2-metric.

Restricting to the non-interacting case of the large deviation result in \cite{BDF12}, it is not clear to the author how the strong assumption of having of having a core relates to strong uniqueness of solutions to the $N$-particle model
\begin{equation*}
\dd X^{N,i}(t) = b(X^{N,i}(t)) \dd t + \sigma(X^{N,i}(t)) \dd W^i(t), 
\end{equation*} 
which is one of the assumptions in \cite{BDF12}. Note that in this paper the large deviation principle is established in the more complex space $\cP(D_{\bR^d}([0,T]))$ as opposed to in $D_{\cP(\bR^d)}([0,T])$ or $D_{\cP(\bR^d)}(\bR^d)$ in \cite{DG87,FK06} and this paper.

Also the restriction of the results to the non-interacting case in \cite{FaMa14} hold on the space $\cP(D_{\bR^d}[0,T])$. Compared to \cite{BDF12} only weak uniqueness is necessary, but on the other hand it is assumed that the diffusion matrices are non-degenerate.

\smallskip 

The results of \cite{DG87,FK06,BDF12,FaMa14} are all more general in the sense that they all hold also for weakly interacting systems.

\begin{proof}[Proof of Theorem \ref{theorem:condition_diffusion}]
By Lemma \ref{lemma:properties_test_functions}, we only have to check Condition \ref{condition:topology_on_D} (d) and (e). By definition of the topology $\tau$ on $C_c^\infty(\bR^d)$, (d) is clear. For (e), define the $\tau$-continuous and convex functions
\begin{multline*}
|f|_{n} := d \sup_{x \in K_n}\sup_i |b^i(x)| \left|\frac{\partial f}{\partial x^i}(x)\right| \\
+ \; \frac{d^2}{2} \sup_{x \in K_n} \sup_{i,j} |a^{ij}(x)| \left(\left|\frac{\partial^2 f}{\partial x^i \partial x^j}(x)\right|+ \left|\frac{\partial f}{\partial x^i}(x)\right| \left|\frac{\partial f}{\partial x^j}(x)\right| \right),
\end{multline*}
and the set $\cN := \left\{f \in C_c^\infty(\bR^d) \, \middle| \, \forall \, n \geq 1, \text{ we have } |f|_n \leq 1\right\}$. Clearly, $\cN$ is closed, convex and balanced. Because $\cN$ is balanced, convex, and $0 \in \cN$, it follows that to prove that $\cN$ is absorbing, it is sufficient to prove that for every $f \in C_c^\infty(\bR^d)$ there exists $\alpha > 0$ such that $\alpha f \in \cN$.

\smallskip

Consider $f \in C_c^\infty(\bR^d)$. By $\tau$-continuity, we obtain that if $\alpha \rightarrow 0$, then $|\alpha f|_n \rightarrow 0$ for all $n \geq 1$. Because there is some $m$ such that $f \in C^\infty_0(K_m)$, we conclude that for $n \geq m$ we have $|f|_{n} = |f|_m$. Thus, we find that $\sup_n |\alpha f|_n \rightarrow 0$ as $\alpha \rightarrow 0$. We conclude that there is some $\alpha > 0$ such that $\alpha f \in \cN$ and, hence,  that $\cN$ is a barrel.

\smallskip

Consider $c \geq 0$, we prove that $\sup_{f \in c \cN} \vn{Hf} < \infty$. Pick some $g \in \cN$. Then there is some $m$ such that $g \in C_0^\infty(K_m) \cap c\cN$. Thus
\begin{equation*}
\vn{Hg} \leq |cg|_m \leq (1 \vee c)^2 |g|_m \leq (1\vee c)^2.
\end{equation*}
Thus, $\sup_{f \in c \cN} \vn{Hf} \leq (1 \vee c)^2$, which establishes Condition \ref{condition:topology_on_D} (e).
\end{proof}

\subsection{A representation of the Lagrangian in terms of an inverse Sobolev space norm} \label{subsec:rep_of_Lagrangian_in_inverse_sobolev_norm}

Let $(x^1,\dots, x^d)$ be the global Euclidean coordinates. For $f \in C_c^\infty(\bR^d)$, define the Riemannian gradient induced by the diffusion matrices $a^{ij}(\cdot)$:
\begin{equation*}
(\nabla f)^i = \sum_{j=1}^d a^{ij}(\cdot) \frac{\partial f}{\partial x^j},
\end{equation*}
for $i \in \{1,\dots,d\}$. If the matrices would be positive definite, the matrices $a^{ij}$ would be invertible with inverses $a_{ij}$. The associated Riemannian inner product for tangent vectors in $T_x \bR^d$ would be
\begin{equation*}
[X,Y]_x = \sum_{i,j=1}^d a_{ij}(x) X^i Y^j,
\end{equation*}
which induces a norm on the tangent space $T_x \bR^d$: $|X|_x = \sqrt{[X,X]_x}$. We conclude that the norm of the gradient of $f$ equals
\begin{equation*}
|\nabla f|_x^2 = \sum_{i,j = 1}^d a^{ij}(x) \frac{\partial f}{\partial x^i}(x)\frac{\partial f}{\partial x^j}(x),
\end{equation*}
which is a formula that also makes sense if the matrix $a^{ij}$ is degenerate.

Define the semi-norm $\vn{f}_\mu^2 := \ip{|\nabla f|^2}{\mu}$ and the Sobolev space $H^1(\mu,\nabla)$, by identifying all functions $f \in C_c^\infty(\bR^d)$ such that $\vn{f-g}_\mu = 0$, and then completing it by using the norm $\vn{\cdot}_\mu$. 
For $\alpha \in C_c^\infty(\bR^d)'$, define the dual norm
\begin{equation*}
\vn{\alpha}_{-1,\mu} = \sup_{\substack{f \in C_c^\infty \\ \vn{f}_\mu \leq 1}} \ip{f}{\alpha} = \sup_{f \in C_c^\infty(\bR^d)} \left\{\ip{f}{\alpha} - \frac{1}{2}\vn{f}_\mu^2 \right\}.
\end{equation*}

The next proposition shows the connection between Theorem \ref{The:LDP2} and Theorem 4.5 by Dawson and Gärtner \cite{DG87}.
\begin{proposition} \label{proposition:representation_of_rate_function_diffusions}
Let $\mu \in \cP(E)$ and let $\alpha \in C_c^\infty(\bR^d)'$, then $\cL(\mu,\alpha) = \frac{1}{2} \vn{\alpha - A'(\mu)}^2_{-1,\mu}$.
\end{proposition}

\begin{proof}
Pick $\mu \in \cP(\bR^d)$ and $\alpha \in C_c^\infty(\bR^d)'$. By \eqref{equation:definition_H_diffusion}, we find
\begin{equation*}
\cL(\mu,\alpha)   = \sup_{f \in C_c^\infty(\bR^d)} \left\{\ip{f}{\alpha} - \ip{Af}{\mu} -\frac{1}{2}\ip{|\nabla f|^2}{\mu}  \right\}  = \frac{1}{2} \vn{\alpha - A'(\mu)}_{-1,\mu}^2.
\end{equation*}
\end{proof}

\subsection{Lévy processes} \label{subsec:Levy}

For our second example, we consider the large deviations of Lévy processes. Let $\PR \in D_{\bR^d}(\bR^+)$ be the law of a $\bR^d$ valued Lévy process. Because the law of $X(1)$ is infinitely divisible, the Lévy-Khintchine formula, Theorem 8.1 in \cite{Sa99}, and Corollary 11.6, \cite{Sa99}, show that there is a one-to-one correspondence between Lévy processes, the law of the process at time $1$, and triplets $(a,\nu,\gamma)$. Here $a \in S^d$, $\nu$ be a non-negative Borel measure on $\bR^d$ satisfying
\begin{equation} \label{eqn:Levy_conditions_on_jump_measure}
\nu(\{0\}) = 0, \qquad \int (|x|^2 \wedge 1) \nu(\dd x) < \infty, \qquad \gamma \in \bR^d.
\end{equation}

\begin{theorem}[Theorem 31.5 \cite{Sa99}]
Let $X$ be a L\'{e}vy process on $\bR^d$ with generating triplet $(a,\nu,\gamma)$. Then $X$ is a strong Feller process with strongly continuous semigroup $\{S(t)\}_{t \geq 0}$ on $C_0(\bR^d)$ and the generator $(A,\cD(A))$ of $\{S(t)\}_{t \geq 0}$ satisfies $C_0^2(\bR^d) \subseteq \cD(A)$ and for $f \in C_0^2(\bR^d)$:
\begin{multline*}
Af(x) = \frac{1}{2} \sum_{i,j = 1}^d a^{ij} \frac{\partial^2 f}{\partial x^i \partial x^j}(x) + \sum_{i = 1}^d \gamma_i \frac{\partial f}{\partial x^i}(x) \\
+ \int\left(f(x+y) - f(x) - \sum_{i = 1}^d y_i \frac{\partial f}{\partial x^i}(x) \bONE_D(y) \right) \nu(\dd y),
\end{multline*}
where $D = \{x \, | \, |x| \leq 1\}$. The set $C_c^\infty(\bR^d)$ is a core for $(A,\cD(A))$.
\end{theorem}

The assumption on $\nu$ is necessary for the integral in the definition of $A$ to be well-defined. For $f \in C_c^\infty(\bR^d)$ with support in a compact set $K \subseteq \bR^d$, a second-order Taylor expansion of the $f(x+y)$ in $y$ yields that
\begin{align*}
\left|f(x+y) - f(x) - \sum_{i = 1}^d y_i \frac{\partial f}{\partial x^i}(x)\right| & \leq \frac{1}{2}\sum_{i,j = 1}^d \left|y_i y_j \frac{\partial^2 f}{\partial x^i \partial x^j}(x)\right| + \theta(\vn{f}_{3,K}) R(y) \\
& \leq \frac{\vn{f}_{2,K}}{2}d |y|^2 + \theta(\vn{f}_{3,K}) |y|^2 \frac{R(y)}{|y|^2},
\end{align*}
where $\theta$ is some continuous non-negative function and where $y \mapsto R(y)$ is continuous, non-negative, and satisfies $\lim_{|y| \rightarrow 0} R(y)|y|^{-2} = 0$. Thus, by \eqref{eqn:Levy_conditions_on_jump_measure}, the integral in the definition of $A$ is well-defined. Additionally, it shows that $f \mapsto Af$ is continuous from $(C_c^\infty(\bR^d))$ to $(C_0(\bR^d), \vn{\cdot})$.

A straightforward calculations shows that for $f \in C_c^\infty(\bR^d)$
\begin{multline*}
Hf(x) = \frac{1}{2} \sum_{i,j = 1}^d a^{ij} \frac{\partial^2 f}{\partial x^i \partial x^j}(x) + \sum_{i = 1}^d \gamma_i \frac{\partial f}{\partial x^i}(x) + \frac{1}{2} \sum_{i,j} a^{ij} \frac{\partial f}{\partial x^i}(x) \frac{\partial f}{\partial x^j}(x) \\
+ \int\left(e^{f(x+y) - f(x)} - 1 - \sum_{i = 1}^d y_i \frac{\partial f}{\partial x^i}(x) \bONE_D(y) \right) \nu(\dd y).
\end{multline*}
An argument based on the Taylor expansion of the exponential shows that also here the integral is well-defined. For $f \in C_c^\infty(\bR^d)$ supported on the compact set $K$:
\begin{equation*}
\left|e^{f(x+y) - f(x)} - 1 - \sum_{i = 1}^d y_i \frac{\partial f}{\partial x^i}(x)\right| \leq \vn{f}_{2,K}d |y|^2 + \theta_2(\vn{f}_{3,K}) |y|^2 \frac{R(y)}{|y|^2},
\end{equation*}
where $\theta_2$ is a non-negative continuous function and $R$ is the same non-negative continuous function as above.

\begin{theorem} \label{theorem:condition_Levy}
$(C_c^\infty(\bR^d),\tau)$ and $H$ satisfy Conditions and \ref{condition:D_algebra_closed_smooth} and \ref{condition:topology_on_D}. As a consequence, Theorems \ref{The:LDP1} and \ref{The:LDP2} hold for iid copies of a Lévy process.
\end{theorem}

To find a suitable barrel for Condition \ref{condition:topology_on_D} (e), note that $\left|e^{a - b} - 1\right| \leq \left|a-b \right| e^{|a-b|}$.

\begin{proof}
By Lemma \ref{lemma:properties_test_functions}, we are left to verify Condition \ref{condition:topology_on_D} (d) and (e). (d) we already proved above. For (e), we define the $\tau$-continuous and convex functions
\begin{multline*}
|f|_{n} := |\gamma|\vn{f}_{1,K_n} + 2\vn{f}_{2,K_n} \sum_{ij} a^{ij} + \int_{D^c} 2\vn{f}_{0,K_n} e^{2 \vn{f}_{0,K_n}}  \, \nu(\dd y) \\
+ \int_D  |y|^2\left[d\vn{f}_{2,K_n} + \theta_2(\vn{f}_{3,K_n}) \frac{R(y)}{y} \right] \nu(\dd y)
\end{multline*}
and the set
\begin{equation*}
\cN := \left\{f \in C_c^\infty(\bR^d) \, \middle| \, \forall \, n \geq 1, \text{ we have } |f|_n \leq 1\right\}.
\end{equation*}
The proof that $\cN$ is a barrel follows as in the proof of Theorem \ref{theorem:condition_diffusion}.

Consider $c \geq 0$ and pick some $g \in c\cN$. Then, there is some $m$ such that $g \in C_0^\infty(K_m) \cap c\cN$. As in the proof of Theorem \ref{theorem:condition_diffusion}, we have
\begin{equation*}
\vn{Hg} \leq |cg|_m \leq \Gamma_{\nu}(c),
\end{equation*}
where $\Gamma_\nu : [0,\infty) \rightarrow [0,\infty)$ is some increasing function that depends on $\theta_2$, and $\nu$. Thus, $\sup_{f \in c \cN} \vn{Hf} \leq \Gamma_{\nu}(c)$, which proves Condition \ref{condition:topology_on_D} (e).
\end{proof}

The relevant information for the bound in terms of $\Gamma_\nu$ is that $\nu$ is independent of the coordinate $x$. We will see below that if the jump kernel $\nu$ is $x$ dependent and unbounded, this does not work.

\subsection{Markov pure jump process} \label{section:example_Markov_jump_processes}

We consider Markov processes on a locally compact separable metric space with a generator of the form
\begin{equation*}
Af(x) = \lambda(x) \int \left[f(y) - f(x)\right] \mu(x,\dd y),
\end{equation*}
where $\mu$ is a transition function from $E$ to $E$. In the setting that $\lambda = 1$ and that $x \mapsto \mu(x,\cdot)$ is continuous from $(E,d)$ to $(\cP(E),weak)$, it is immediate that $A$ generates a Feller process and that Conditions \ref{condition:D_algebra_closed_smooth} and \ref{condition:topology_on_D} are satisfied with $(D,\tau_D) = (C_0(E),\vn{\cdot})$. We summarize this as a proposition.

\begin{proposition}
Theorem \ref{The:LDP2} holds with $(D,\tau_D) = (C_0(E),\vn{\cdot})$ for Markov jump processes on a locally compact space if $\lambda$ is bounded.
\end{proposition}

The condition that $\lambda$ is bounded corresponds to Assumption 1 in \cite{DjKa95}, which is proven for jump processes on $\bR^d$, but with weak time-dependent interaction. 
A representation or the rate function in terms of Orlicz norms like in Theorem 3.1 in \cite{Le95}. We will not reproduce the proof here as it significantly longer than the counterpart for diffusion processes. Also see \cite{Le95b,Leo01} for more details on this proof.

\smallskip

In the setting that the kernel $\lambda$ is unbounded needs more care. One such setting is treated below in Section \ref{subsec:IPS}. In the setting of Theorem 8.3.1 in \cite{EK86}, however, our method seems to fail even if we consider an LF topology like in Section \ref{section:diffusion_processes}.

\subsection{Interacting particle systems} \label{subsec:IPS}

To show that our approach works in a wide range of contexts, we consider also consider interacting particle systems as defined in \cite{Li85}. Let $W$ be a compact metric space and let $S$ be a countable set. Define $(E = W^S ,d)$, the product space with $d$ a metric that is compatible with the topology, on which we will define a Markov process $\left\{\eta(t) \right\}_{t \geq 0}$. Examples are the exclusion process, the contact process, etcetera. We follow the construction in Liggett \cite{Li85}.

\smallskip

For $\Lambda$ a finite subset of $S$ and $\zeta \in W^\Lambda$ let $c_\Lambda(\eta,\dd \zeta)$ be the rate at which the system makes a transformation from configuration $\eta$ to $\eta^\zeta$ which is defined by
\begin{equation*}
\eta^\zeta_x =
  \begin{cases}
   \eta_x   & \text{if } x \notin \Lambda, \\
   \zeta_x  & \text{if } x \in \Lambda.
  \end{cases}
\end{equation*}
We assume that $c_\Lambda(\eta,\dd \zeta)$ is weakly continuous in the first variable. Because the total rate of jumps could be infinite, we need to specify a class of test functions for the generator. For $f \in C(E)$, define
\begin{equation*}
\Delta_f(x) = \sup \left\{|f(\eta) - f(\zeta)| \, \middle| \, \text{ for } y \neq x: \; \eta_y = \zeta_y \right\}
\end{equation*}
the variation of $f$ at $x \in S$. The natural space of test functions is given by
\begin{equation*}
D = \left\{f \in C_b(E) \, \middle| \, \tn{f} := \sum_{x \in S} \Delta_f(x) < \infty \right\}. \label{eqn:defD}
\end{equation*}
For functions $f \in D$, define the formal generator $A$ to be
\begin{equation} \label{eqn:defL}
Af(\eta) = \sum_\Lambda \int c_\Lambda(\eta,\dd \zeta) \left[f(\eta^\zeta) - f(\eta)\right].
\end{equation}
Note that the total jump-rate in this generator could be infinite. Theorem I.3.9 in \cite{Li85} gives conditions under which $A$ generates a Feller semigroup $\{S(t)\}_{t \geq 0}$ and Markov process $ (\eta(t))_{t \geq 0}$. One of the main conditions in this theorem is
\begin{equation} \label{assumption:Liggett3.3}
|A|_D := \sup_x \sum_{\Lambda \ni x} c_\Lambda   < \infty,
\end{equation}
where $c_\Lambda = \sup\{c_\Lambda(\eta,W^\Lambda) \, | \, \eta \in E\}$ is the maximal total variation of $c_\Lambda(\eta,\cdot)$. This condition implies that the sum in \eqref{eqn:defL} is uniformly convergent and that for $f\in D$:
\begin{equation} \label{assumption:Liggett3.3_corollary}
\vn{Af} \leq |A|_D \tn{f}.
\end{equation}

The same condition implies that $A^g$ and $H$ are well defined for $f,g \in D$. Analogously to the operators for jump processes, we find
\begin{align*}
A^g f(\eta) & =  \sum_\Lambda \int c_\Lambda(\eta,\dd \zeta) e^{g(\eta^\zeta) - g(\eta)} \left[f(\eta^\zeta) - f(\eta)\right], \\
H f(\eta) & = \sum_\Lambda \int c_\Lambda(\eta,\dd \zeta) \left[ e^{g(\eta^\zeta) - g(\eta)} -1 \right].
\end{align*}

Motivated by \eqref{assumption:Liggett3.3_corollary}, our goal is to equip $D$ with a topology $\tau_D$ based on the semi-norm $\tn{\cdot}$. Note that $\tn{\bONE} = 0$, so $\tn{\cdot}$ alone can not define a Hausdorff topology. Only the constants, however, have this property. Thus, let $\tau_D$ be the topology induced by $\vn{\cdot}_D := \tn{\cdot} + \vn{\cdot}$.

\begin{theorem} \label{theorem:IPS_D_satisfies_conditions}
Let $A$ satisfy the conditions of Theorem I.3.9 in \cite{Li85}, amongst those including \eqref{assumption:Liggett3.3}.

Then $(D,\vn{\cdot}_D)$ satisfies Conditions \ref{condition:D_algebra_closed_smooth} and \ref{condition:topology_on_D}. As a consequence, Theorems \ref{The:LDP1} and \ref{The:LDP2} hold for i.i.d. copies of interacting particle processes with generator $A$.
\end{theorem}


\begin{proof}
A straightforward argument, using the density of local functions establishes the separability of $(D,\tau_D)$, implying Condition \ref{condition:topology_on_D} (a). (b) is immediate by the definition of $\tau_D$.  Conditions \ref{condition:D_algebra_closed_smooth} and \ref{condition:topology_on_D} (c) follows from a number of straightforward calculations using the semi-norm $\tn{\cdot}$.

Condition \ref{condition:topology_on_D} (d) was obtained in \eqref{assumption:Liggett3.3_corollary}. For (e), fix $f \in D$, then the function $\alpha \mapsto e^\alpha$ defined on $[-2\vn{f}_Q,2\vn{f}_Q]$ is Lipschitz continuous, with Lipschitz constant $e^{2\vn{f}_Q}$. This means that $|e^\alpha - 1 | \leq |\alpha|  e^{2\vn{f}_Q}$ on the appropriate domain. Applying this to $\vn{Hf}$, we obtain
\begin{equation*}
\vn{Hf} \leq e^{2\vn{f}_Q} \sum_\Lambda \left|\int c_\Lambda(\eta, \dd \zeta) \left[ f(\eta^\zeta) - f(\eta)\right] \right| \leq e^{\tn{f}} \tn{f} \sup_x \sum_{\Lambda \ni x} c_\Lambda.
\end{equation*}
It follows $\cN = \{f \in D \, | \, \tn{f} \leq 1\}$ satisfies the condition for (e).
\end{proof}

We proceed with a short discussion on giving a respresentation for $D'$ in terms of set functions. Because we can always choose $\cN$ in Condition \ref{condition:topology_on_D} such that it contains all constant functions, we can restrict our attention to $(D/\cC)'$, where $\cC$ is the space of constant functions. 

\smallskip

We introduce some notation. For $\Lambda \subseteq S$, let $\cE_\Lambda : = \sigma(\eta_x \, | \, x \in \Lambda)$. Furthermore, $\tilde{\Pi}$ is the space of additive set functions $\alpha$ on the algebra $\cE_a := \bigcup_{\Lambda : |\Lambda| < \infty}  \cE_\Lambda$, for which it holds that $\alpha(E) = 0$. Note that the $\sigma$-algebra $\cE$ is given by $\sigma(\cE_a)$.

For $\alpha \in \tilde{\Pi}$ and a finite subset $\Lambda \subseteq S$, we denote the restriction of $\alpha$ to $\cE_\Lambda$ by $P_\Lambda \alpha$ and we set $P_x := P_{\{x\}}$. Also, we define the function $\vn{\alpha}_\Pi = \sup_x \vn{P_x \alpha}_{TV}$ taking values in $[0,\infty]$.
\begin{definition}
Let $\Pi$ be the set
\begin{equation*}
\Pi := \left\{\alpha \in \tilde{\Pi} \, \middle| \, \vn{\alpha}_\Pi < \infty \right\}.
\end{equation*}
\end{definition}

It follows that $\Pi$ is a vector space and that $\vn{\cdot}_\Pi$ is a norm on $\Pi$. We have the following results on $(D/\cC)'$ and $\Pi$, the proof of which is tedious, but straightforward.

\begin{proposition}\label{proposition:Representationtheorem}
$(\Pi,\vn{\cdot}_\Pi)$ is a Banach space and we have that $((D/\cC)', \tn{\cdot})$ is isometrically isomorphic to $(\Pi, \frac{1}{2} \vn{\cdot}_\Pi)$.
Additionally, we have $\vn{P_\Lambda \alpha}_{TV} \leq |\Lambda| \vn{\alpha}_\Pi$ for all finite subsets $\Lambda \subseteq S$.
\end{proposition}

\section{Appendix: The large deviation principle via Sanov's theorem} \label{section:LDP_on_path_space}

In this Appendix, we sketch how to prove Theorem \ref{The:LDP1}. In this setting, we let $(E,d)$ be a complete separable metric space. We prove the large deviation principle for a general class of processes via Sanov's theorem and the contraction principle. A similar approach has been taken in Lemma 4.6 of \cite{DG87}. More care needs to be taken as the contraction map $\phi$ defined by $\phi : \cP(D_E(\bR^+)) \rightarrow D_{\cP(E)}(\bR^+)$ is not continuous, whereas in the context of $C_E(\bR^+)$ it is. 

\smallskip

Define for every $t$ the measurable maps $\pi_t, \pi_{t-} : D_E(\bR^+) \rightarrow E$ by $\pi_t (x) := x(t)$ and $\pi_{t-}(x) = x(t-)$, see e.g. Proposition III.7.1 in \cite{EK86}.

Let $\PR$ be a probability measure on $D_E(\bR^+)$, and let $X = (X(t))_{t \geq 0}$ be the process with law $\PR$. Define $\mu(t) = \PR \circ \pi_t^{-1}$ and $\mu(t-) = \PR \circ \pi_{t-}^{-1}$ the laws of $X(t)$ and $X(t-)$. Also define the map $\phi : \mathcal{P}(D_E(\bR^+)) \rightarrow \mathcal{P}(E)^{\bR^+}$ by setting $\phi(\PR) = (\mu(t))_{t \geq 0}$ and finally define the maps $\phi_t : \mathcal{P}(D_E(\bR^+)) \rightarrow \mathcal{P}(E)$ by setting $\phi_t(\PR) = \mu(t)$. 

As mentioned above, the map $\phi$ is not continuous. Discontinuity problems can be avoided by additionally assuming that the process is continuous in probability.

\begin{proposition} \label{proposition:phiiscontinuous}
$\phi : \mathcal{P}(D_E(\bR^+)) \rightarrow D_{\mathcal{P}(E)}(\bR^+)$ is measurable and, additionally, continuous at measures $\PR$ for which it holds that for every $t > 0$: $\PR[X(t) = X(t-)] = 1$.
\end{proposition}
A similar statement for the finite dimensional projections $\phi_t$, can be found in \cite[Theorem 3.7.8]{EK86}.

To prove measurability of $\phi$, we first prove a useful lemma.

\begin{lemma} \label{lemma:pushforward_is_measurable}
Let $E,F$ be Polish spaces and let $\xi : E \rightarrow F$ be measurable. Then the map $\xi_* : \cP(E) \rightarrow \cP(F)$ defined by $\xi_*\mu = \mu \circ \xi^{-1}$ is measurable for the Borel $\sigma$-algebras with respect to the weak topology on $\cP(E)$ and $\cP(F)$.
\end{lemma}

\begin{proof}
Let $H$ be a Polish space, e.g. $E$ or $F$. We start by proving that the Borel $\sigma$-algebra $\cB_w$ for the weak topology on $\cP(H)$ equals the Borel $\sigma$-algebra $\cB_{TV}$ for the total variation norm on $\cP(H)$. As the total variation topology is finer than the weak topology, we find $\cB_w \subseteq \cB_{TV}$. As the total variation norm is lower semi-continuous with respect to the weak topology, cf. Theorem 7.9.1 in \cite{Bo07} which identifies the total variation topology as the strong topology corresponding on $\cM(H)$ induced by $C_b(H)$, we find that $\cB_{TV} \subseteq \cB_w$.

\smallskip

The result follows immediately from this identification as the map $\xi_*$ is continuous for the total variation topologies.
\end{proof}

\begin{proof}[Proof of Proposition \ref{proposition:phiiscontinuous}]
To prove the measurability of $\phi$, Proposition 3.7.1 in \cite{EK86} implies that the Borel $\sigma$-algebra on $D_{\cP(E)}(\bR^+)$ is generated by the inverse images of Borel sets in $\bR$ of a collection of maps $\iota^f_s : D_E(\bR^+) \rightarrow \bR$ defined by $\iota^f_s(\{\mu(t)\}_{t \geq 0}) = \int f \dd \mu(s)$ for all $f \in C_b(E)$ and $s \in \bR^+$. Thus, it suffices to prove that $\iota_s^f \circ \phi$ is measurable for all $s$ and $f$.

We can write $\iota_s^f \circ \phi$ as $\iota^f \circ (\pi_s)_*$, where $\pi_t : D_E(\bR^+) \rightarrow E$ was defined as $\pi_t(x) = x(t)$ and where $\iota^f : \cP(E) \rightarrow \bR$ is defined as $\iota^f(\mu) = \int f \dd \mu$. By Lemma \ref{lemma:pushforward_is_measurable} $(\pi_t)_*$ is measurable, and by definition $\iota^f$ is continuous. We conclude that $\iota_s^f \circ \phi= \iota^f \circ (\pi_s)_*$ is measurable for all $s \geq 0$ and $f \in C_b(E)$. We conclude that $\phi$ is measurable.
	
\smallskip
	
We proceed with the proof of continuity. Let $\PR^n, \PR \in \mathcal{P}(D_E(\bR^+))$ such that $\PR^n \rightarrow \PR$ weakly and $\PR$ such that for every $t$ $\PR[X(t) = X(t-)] = 1$. By the Skorokhod representation Theorem \cite[Theorem 3.1.9]{EK86}, we can find a probability space $\left(\Omega, \mathcal{F}, P\right)$ and $D_E(\bR^+)$ valued random variables $Y^n, Y$ distributed as $X^n$ and $X$ under $\PR^n, \PR$ such that $Y^n \rightarrow Y \; P$ a.s. 

Let $\left\{ t_n\right\}_{n \geq 0}$ be a sequence converging to $t > 0$. Define the sets
\begin{equation*}
A := \left\{Y(t) = Y(t-) \right\}, \quad \, B := \left\{d(Y^n(t_n),Y(t)) \wedge d(Y^n(t_n),Y(t-)) \rightarrow 0 \right\}.
\end{equation*}
By the assumption that $\PR[X(t) = X(t-)] = 1$, it follows that $P[A] = 1$. By Proposition 3.6.5 in \cite{EK86}, and the fact that $Y^n \rightarrow Y \; P$ a.s. it follows that $P[B] = 1$. Combining these statements yields $P\left[Y^n(t_n) \rightarrow Y(t)\right] \geq P[A \cap B] = 1$, which implies that $\mu^n(t_n) \rightarrow \mu(t)$. As $\mu(t) = \mu(t-)$ by assumption, Proposition 3.6.5 in Ethier and Kurtz yields the final result.
\end{proof}

\subsection{Large deviations for measures on the Skorokhod space}

Suppose that we have a process $X$ on $D_E(\bR^+)$ and a corresponding measure $\PR \in \mathcal{P}(D_E(\bR^+))$. Then Sanov's theorem, Theorem 6.2.10 in \cite{DZ98}, gives us the large deviation behaviour of the empirical distribution $L_n^X$ of independent copies of the process $X$ $X^1,X^2,\dots$:
\begin{equation*}
L_n^X := \frac{1}{n} \sum_{i=1}^n \delta_{\{X^i\}} \in \mathcal{P}(D_E(\bR^+))
\end{equation*}
with as the rate function the relative entropy defined in \eqref{definition:relative_entropy}. As a consequence of the contraction principle and Proposition \ref{proposition:phiiscontinuous}, we obtain the following result.

\begin{theorem} \label{The:contract}
Suppose that $\PR$ satisfies $\PR[X(t) = X(t-)] = 1$ for every $t \geq 0$, then the large deviation principle holds for 
\begin{equation*}
\left(L_n^{X(t)}\right)_{t \geq 0} = \left( \frac{1}{n} \sum_{i=1}^n \delta_{X^i(t)} \right)_{t \geq 0}
\end{equation*}
on $D_{\mathcal{P}(E)}(\bR^+)$ with rate function
\begin{equation*}
I((\nu_t)_{t \geq 0}) = \inf \{H(\mathbb{Q} \, | \, \PR) \, | \, \mathbb{Q} \in \mathcal{P}(D_E(\bR^+)), \phi(\mathbb{Q}) = (\nu(t))_{t\geq0} \}
\end{equation*}
and $I$ is finite only on $C_{\cP(E)}(\bR^+)$.
\end{theorem}

\begin{proof}
The measures $\mathbb{Q}$ for which it holds that $I(\mathbb{Q}) < \infty$ satisfy $\mathbb{Q} << \PR$ hence it follows that for every $t$: $\mathbb{Q}[X(t) = X(t-)] = 1$. This yields that $\phi$ is continuous at $\mathbb{Q}$ by Proposition \ref{proposition:phiiscontinuous}. 

By the contraction principle, Theorem 4.2.1 and remark (c) after Theorem 4.2.1 in \cite{DZ98}, we obtain the large deviation principle on $D_{\mathcal{P}(E)}(\bR^+)$ with $I$ as given in the theorem.
\end{proof}

\subsection{The large deviation principle for Markov processes} \label{subsec:LDPforMarkov}

Although Theorem \ref{The:contract} can be applied to a wide range of (time-inhomogeneous) processes, we explore its consequences for time-homogeneous Markov processes. 

\begin{lemma} \label{lemma:ldpholdsonskorokhodmarkov} 
Suppose that the process $X$ with corresponding measure $\PR$ on $D_E(\bR^+)$ solves the martingale problem for $(A,\cD(A))$ with starting measure $\PR_0$. Then, it holds that for every $t \geq 0$ $\PR[X(t) = X(t-)] = 1$. Hence, the large deviation principle holds for $\{L_n^{X(t)} \}_{t \geq 0}$ on $D_{\mathcal{P}(E)}(\bR^+)$ with rate function
\begin{equation*}
I((\nu_t)_{t \geq 0}) = \inf \{H(\mathbb{Q} \, | \, \PR) \, | \, \mathbb{Q} \in \mathcal{P}(D_E(\bR^+)), \phi(\mathbb{Q}) = (\nu(t))_{t\geq0} \}
\end{equation*}
and $I$ is finite only on $C_{\cP(E)}(\bR^+)$.
\end{lemma} 

\begin{proof}
To apply Theorem \ref{The:contract}, we need to check that $\PR[X(t) = X(t-)] = 1$ for every $t \geq 0$, but this follows by Theorem 4.3.12 in \cite{EK86}.
\end{proof}

Using this result, Theorem \ref{The:LDP1} follows without much effort.

\begin{proof}[Proof of Theorem \ref{The:LDP1}]
The large deviation principle follows by Lemma \ref{lemma:ldpholdsonskorokhodmarkov}. This lemma also gives that the rate function is $\infty$ on the complement of $C_{\mathcal{P}(E)}(\bR^+)$. 

Next, we consider the form of the rate function. By Sanov's theorem and the contraction principle, the vector $(L_n^{X(0)},L_n^{X(t_1)},\dots,L_n^{X(t_k)})$ satisfies the large deviation principle on $\cP(E)^{k+1}$ with some rate function $I[0,t_1,\dots,t_k]$.  The representation
\begin{align*}
I[t_0,\dots, t_k](\nu(0),\dots, \nu(t_k)) & = H(\nu(0) \, | \mu(0)) + \sum_{i=1}^k I_{t_i- t_{i-1}}(\nu(t_i) \, | \, \nu(t_{i-1})), \\
I_t(\pi_1 \, | \, \pi_2) & = \sup_f \ip{f}{\pi_1} - \ip{V(t)f}{\pi_2},
\end{align*}
is proven for diffusion processes in Theorem 3.5 and Lemma 4.7 in \cite{DG87}. The proofs of these results only use the Feller property of the transition kernels, and can thus be generalized without any problems to this more general context.

To obtain the rate function $I$ as a supremum over rate functions for  $I[0,t_1,\dots,t_k]$ for finite dimensional problems
\begin{equation*}
I(\nu) = \begin{cases}
\sup_{0,t_1,\dots,t_k} I[0,t_1,\dots,t_k](\nu(0),\nu(t_1)\dots,\nu(t_k)) & \text{if } \nu \in C_{\mathcal{P}(E)}(\bR^+), \\
\infty & \text{otherwise},
\end{cases}
\end{equation*}
we use Theorem 4.13 and Theorem 4.30 in Feng and Kurtz\cite{FK06}. 
\end{proof}

\section{Appendix: Souslin spaces, barrelled spaces, and Gelfand integration} \label{section:FAappendix}

\subsection{Barrelled spaces and Gelfand integration}

Barrelled spaces, introduced in the paragraph following Equation \eqref{eqn:polar_def} on page \pageref{eqn:polar_def}, are of importance because they allow for integration theory on the dual of the space.

\smallskip

Let $(\Omega, \cF, \mu)$ be a complete and finite measure space, and let $\cX$ be a barrelled space with continuous dual $\cX'$. We equip $X'$ with $\sigma(\cX',\cX)$, the weak* topology.
\begin{definition}
A function $f : \Omega \rightarrow \cX'$ is called weak* measurable if the scalar function $\omega \mapsto \ip{x}{f(\omega)}$ is $\cF$ measurable for every $x \in \cX$. Such a function $f$ is called \textit{Gelfand} or weak* integrable if $\ip{x}{f} \in \cL^1(\Omega,\cF,\mu)$ for every $x \in \cX$.
\end{definition}

For Gelfand integrable functions, we obtain from \cite[pages 52-53]{DU77} combined with the closed graph theorem in \cite[Proposition 7.1.11]{PCB87} or \cite[Theorem VI.7]{RR73}, the following result.

\begin{theorem} \label{theorem:Gelfandrepresentation}
Let $\cX$ be a barrelled space and $(\Omega, \cF, \mu)$ a complete and finite measure space. For every measurable set $A \in \cF$ and Gelfand integrable function $f : \Omega \rightarrow \cX'$, there exists a unique $x_A' \in \cX'$ such that
\begin{equation*}
\ip{x}{x_A'} = \int_A \ip{x}{f(\omega)} \mu(\dd \omega)
\end{equation*}
for all $x \in \cX$. This element $x_A'$ will be denoted by $\int_A f \dd \mu$.
\end{theorem}

\subsection{Souslin spaces}

\begin{definition}
A space $(Y,\tau_Y)$ is called Souslin, if $Y = f(X)$ for some complete separable metric space $(X,\tau_X)$ and some continuous function $f : (X,\tau_X) \rightarrow (Y,\tau_y)$.
\end{definition}

For more background on Souslin spaces, see Chapters 6 and 7 in \cite{Bo07}.

\begin{lemma} \label{lemma:UisSouslin} 
Let $(X,\tau)$ be a separable barrelled locally convex space and $T$ a barrel in $(X,\tau)$. Then $(\bigcup_n n T^\circ, wk^*) \subseteq (X',wk^*)$ is a Souslin space.
\end{lemma}

In particular, as the unit ball in a Banach space $B$ is a barrel, the dual $(B',wk^*)$ of separable Banach space is Souslin.

\begin{proof}
As $(X,\tau)$ is barrelled, $T$ is a neighbourhood of $0$. Consequentially, $T^\circ$ is an equi-continuous set in $(X^*,wk^*)$ by 21.3.(1) in K\"{o}the \cite{Ko69}. By the Bourbaki-Alaoglu theorem, 20.9.(4) \cite{Ko69}, this set is weak* compact. 

Furthermore, by 39.4.(7) in \cite{Ko79}, $T^\circ$ is metrisable. $(T^\circ,wk^*)$ is compact and metric, which implies that it is complete separable metric and as a consequence Souslin. We can do the same for $n \cN^\circ$ for every $n \in \bN$, so we obtain that $(\bigcup_n n \cN^\circ,wk^*)$ is Souslin \cite[Theorem 6.6.6]{Bo07}. 
\end{proof}

\smallskip

\textbf{Acknowledgement}

The author thanks Frank Redig and Jan van Neerven for reading the manuscript and numerous valuable comments. The author also thanks anonymous referees for pointing out a mistake and suggestions on earlier versions of the paper.


\bibliographystyle{plain} 
\bibliography{../KraaijBib}{}

\end{document}